\documentclass[11pt]{article}

\usepackage{amsfonts,amsmath}
\usepackage{latexsym}
\usepackage{bbm}
\usepackage{todonotes}

\usepackage{mathtools}

\usepackage{amsmath}
\usepackage{amssymb}
\usepackage{amsthm}
\usepackage{mathrsfs}
\usepackage{amssymb,url}

\author{K\'aroly J. B\"or\"oczky\footnote{Alfr\'ed R\'enyi Institute of Mathematics, Realtanoda u. 13-15, H-1053 Budapest, Hungary, boroczky.karoly.j@renyi.hu}, 
Shibing Chen\footnote{School of Mathematical Sciences, University of Science and Technology of China, 230026 Hefei, China, chenshib@ustc.edu.cn}, 
Weiru Liu\footnote{School of Mathematics and Statistics, Central China Normal University,  
430079 Wuhan, China, lwr1997@ccnu.edu.cn}, 
Christos Saroglou\footnote{Department of Mathematics, University of Ioannina, Greece, csaroglou@uoi.gr}}
\title{Uniqueness in the near isotropic $L_p$ dual Minkowski problem}

\newcommand{\proofbox}{\mbox{ $\Box$}\\}
\newcommand{\R}{\mathbb{R}}
\newcommand{\E}{\mathbb{E}}

\newcommand{\N}{\mathbb{N}}

\newcommand{\HH}{\mathcal{H}}

\newtheorem{lemma}{LEMMA}[section]
\newtheorem{theo}[lemma]{THEOREM}

\newtheorem{claim}[lemma]{CLAIM}
\newtheorem{coro}[lemma]{COROLLARY}

\newtheorem{prop}[lemma]{PROPOSITION}
\newtheorem{remark}[lemma]{REMARK}
\newtheorem{problem}[lemma]{PROBLEM}

\begin{document}

\maketitle

\begin{abstract}
For $n>1$ and $-1<p<1$, we prove that if $q$ is close to $n$ and the $q$th $L_p$ dual curvature is H\"older close to be the constant one function, then this "near isotropic" $q$th $L_p$ dual Minkowski problem on $S^{n-1}$ has a unique solution. Along the way, we establish a $C^0$ estimate for $-1<p<1$ that is optimal in the sense that if $p<-1$ and $q=n$, then it is known that the analogous $C^0$ estimate does not hold. We also prove the uniqueness of the solution of the near isotropic even  $q$th $L_p$ dual Minkowski problem on $S^{n-1}$ if $-1<p<q<\min\{n,n+p\}$ and $q>0$.
\end{abstract}

\noindent {\bf MSC classes:} 52A20 (35J96)

	
\section{Introduction}
\label{secIntroduction}

In his classical paper about the Gauss curvature flow and ``worn stones", Firey \cite{Fir74} proposed on the one hand the classification of the cone volume measure - that is, the solution of the Monge-Amp\`ere equation \eqref{LpMonge} below, or its "Alexandrov version" \eqref{LpMinkowskiAlexandrov} on $S^{n-1}$ in the case $p=0$, - and on the other hand, asked whether the solution is unique in the isotropic case; namely, when the function $f$ in \eqref{LpMonge} is a positive constant. The problem was extended to the $L_p$ Minkowski problem \eqref{LpMonge} and \eqref{LpMinkowskiAlexandrov} by Lutwak in the 1990's where the classical Minkowski problem is the $p=1$ case. While if $p>1$, then the solution of \eqref{LpMinkowskiAlexandrov} is unique for any 
non-trivial finite measure $\mu$ according to Hug, Lutwak, Yang, Zhang \cite{HLYZ05} and Chou, Wang \cite{ChW06}, 
this property fails if $p<1$:
The existence of multiple solutions of \eqref{LpMinkowskiAlexandrov} was established by Chen, Li, Zhu \cite{CLZ17} if $p\in(0,1)$, Chen, Li, Zhu \cite{CLZ19} if $p=0$, and by  Li, Liu, Lu \cite{LLL22} if $p<0$ (in the letter case, we may assume that the $f$ and $h$ in \eqref{LpMonge} are even and have axial rotational symmetry). In addition, E. Milman \cite{Mil24} provides a systematic study of non-uniqueness in \eqref{LpMonge} when $-n<p<1$. Finally, if $p=-n$, then centered ellipsoids of the same volume are all solutions of the isotropic $L_{-n}$ Minkowski problem \eqref{LpMonge} for the suitable common constant $f$, and they are actually all the solutions according to
 the earlier papers  Calabi \cite{Cal58}, Pogolerov \cite{Pog72}, Cheng, Yau \cite{ChY86}, and the novel approaches
Crasta, Fragal\'a \cite{CrF23}, Ivaki, E. Milman \cite{IvM23} and Saroglou \cite{Sar22}.

However, in the isotropic case of \eqref{LpMonge}, Brendle, Choi, Daskalopoulos \cite{BCD17} established the uniqueness of the solution if $p\in(-n,1)$ (see also Saroglou \cite{Sar22}). Therefore, it is a natural question whether the solution of the $L_p$ Minkowski problem  \eqref{LpMonge} is unique if $f$ is $C^\alpha$ close to a constant function. If $p\in[0,1)$, then the positive answer has been established by Chen, Feng, Liu \cite{CFL22}, B\"or\"oczky, Saroglou \cite{BoS24} and Hu, Ivaki \cite{HuI25} (cf. Theorem~\ref{local-uniqueness-Sp}). In this paper, we extend the uniqueness of the solution of the near isotropic $L_p$ Minkowski problem \eqref{LpMonge} to the case $p\in(-1,0)$ (cf. Theorem~\ref{local-uniqueness-Cpq}), as well. Actually, Theorem~\ref{local-uniqueness-Cpq} is much more general, it discusses the recently found $L_p$ $q$th dual curvature measure (cf. Huang, Lutwak, Yang, Zhang \cite{HLYZ16} and Lutwak, Yang, Zhang \cite{LYZ18}), and the corresponding Monge-Amp\'ere equation \eqref{Lp-dualMink-Monge-Ampere-intro}.

Let us introduce the notions and notations needed to state our results, see Section~\ref{sec-basics} for more detailed discussions.
We call a compact convex set in $\R^n$ with non-empty interior a convex body where we typically assume that $n\geq 2$. The family of convex bodies in $\R^n$ containing the origin $o$ is denoted by $\mathcal{K}^n_o$, and for $K\in \mathcal{K}^n_o$, the support function 
$h_K(w)=\max_{x\in K}\langle w,x\rangle$ for $w\in \R^n$ is convex and $1$-homogeneous.
Let $\partial'K$ denote the subset of the boundary  of a $K\in\mathcal{K}^n_o$ such that there exists a unique exterior unit normal vector
$\nu_K(x)$ at any point $x\in \partial'K$.
It is well-known that $\HH^{n-1}(\partial K\setminus\partial'K)=0$ and $\partial'K$ is a Borel set  (see Schneider \cite{Sch14}) where $\HH^k$ is the $k$-dimensional Hausdorff measure normalized in a way such that it coincides with the Lebesgue measure on $k$-dimensional affine subspaces. We note that for any measure in this paper, the measure of the emptyset is $0$, and to simplify formulas, we also use the notation
$$
|K|=\mathcal{H}^n(K).
$$
The function $\nu_K:\partial'K\to S^{n-1}$ is the spherical Gauss map
that is continuous on $\partial'K$.
 The surface area measure $S_K$ of $K$ is a Borel measure on $S^{n-1}$ 
satisfying that $S_K(\eta)=\HH^{n-1}(\nu_K^{-1}(\eta))$
 for any Borel set $\eta\subset S^{n-1}$. In particular, if $\partial K$ is $C^2_+$, then 
\begin{equation}
\label{SKMonge}
dS_K=\det(\nabla^2 u+u\,I)d\HH^{n-1}
\end{equation}
where $u=h_K|_{S^{n-1}}$, and $\nabla u$ and  $\nabla^2 u$ are the spherical gradient and the Hessian of $u$ with respect to a moving orthonormal frame, and $I$ denotes the identity matrix of suitable size. 

For $K\in \mathcal{K}^n_o$ and $p\in\R$, Lutwak's $L_p$ surface area $S_{p,K}$ is defined by $dS_{p,K}=u^{1-p}dS_K$
for $u=h_K|_{S^{n-1}}$
(see Lutwak \cite{Lut93}, B\"or\"oczky \cite{Bor23}); in particular, $S_K=S_{1,K}$ is the surface area measure and $V_K=\frac1n\,S_{0,K}$ is the so-called cone-volume measure.
If $p\in\R$, then the corresponding $L_p$ Minkowski problem asks for the solution of the Monge-Amp\`ere equation
\begin{equation}
\label{LpMonge}
u^{1-p}\det(\nabla^2 u +u\,{\rm Id})= f
\end{equation}
for a given non-negative  $f\in L_1(S^{n-1})$. More generally, 
given a finite non-trivial measure $\mu$ on $S^{n-1}$, we search  for a $K\in \mathcal{K}^n_o$ satisfying
\begin{equation}
\label{LpMinkowskiAlexandrov}
S_{p,K}= \mu,
\end{equation}
where  $u=h_K|_{S^{n-1}}$ is a solution of \eqref{LpMonge} in the sense of Alexandrov (or in the sense of measure) if
$d\mu=f\,d\HH^{n-1}$. 
The case $p=1$ of \eqref{LpMonge} or \eqref{LpMinkowskiAlexandrov} is the classical Minkowski problem, solved by the work of
Minkowski \cite{Min03,Min11}, Aleksandrov \cite{Ale38a,Ale96},
 Nirenberg \cite{Nir53}, Cheng, Yau \cite{ChY76}, Pogorelov \cite{Pog78}, Caffarelli \cite{Caf90a,Caf90b} stretching throughout the 20th century.
For some existence results about the $L_p$-Minkowski problem without evenness condition for $p\neq 1$, 
see for example Chou,  Wang \cite{ChW06}, Chen, Li, Zhu \cite{CLZ17,CLZ19}, 
Bianchi, B\"or\"oczky, Colesanti, Yang \cite{BBCY19} and
Guang, Li, Wang \cite{LGWa}.

The cone volume measure $V_K=\frac1n\,S_{0,K}$ nicely intertwining with linear transforms (cf. Section~\ref{sec-basics})
was introduced by Firey \cite{Fir74}, and has been a widely used tool since the paper
Gromov, Milman \cite{GromovMilman}, see for example Barthe, Gu\'{e}don, Mendelson, Naor \cite{BG}, Naor \cite{Nar07},  Paouris, Werner \cite{PaW12}.
The still open logarithmic Minkowski problem (\eqref{LpMonge} or \eqref{LpMinkowskiAlexandrov} with $p=0$)  was posed by Firey \cite{Fir74} in 1974, who showed that
if $f$ is a positive constant function and $p=0$ in \eqref{LpMonge}, then \eqref{LpMonge} has a unique even solution coming from the suitable centered ball. 
Without the evenness condition, uniqueness of the solution of \eqref{LpMonge} or \eqref{LpMinkowskiAlexandrov} is known if
\begin{itemize}
\item $p>1$, $p\neq n$ and the measure $\mu$ in \eqref{LpMinkowskiAlexandrov} is not concentrated on any  closed hemisphere (see Hug, Lutwak, Yang, Zhang \cite{HLYZ05}, and also the earlier paper by
Lutwak \cite{Lut93a} for the case when the function $f$ in \eqref{LpMonge} is a positive constant),
\item $-n<p<1$ and the function $f$ in \eqref{LpMonge} is a positive constant (see Andrews \cite{And99} if $n=2,3$, and Brendle, Choi, Daskalopoulos \cite{BCD17} if $n\geq 4$, and also the related papers or different arguments by
Andrews, Guan, Ni \cite{AGN16}, Saroglou \cite{Sar22} and Ivaki, Milman \cite{IvM23}).
\end{itemize}
As we have discussed before, the uniqueness of the solution of \eqref{LpMinkowskiAlexandrov} may not hold if $p<1$.
 Still, it is a fundamental question whether \eqref{LpMonge} has a unique solution if $-n<p<1$
and $f$ is close to be a constant function. In this respect, Chen, Feng, Liu \cite{CFL22}, B\"or\"oczky, Saroglou \cite{BoS24} and Hu, Ivaki \cite{HuI25} prove the following.

\begin{theo}[\cite{BoS24}, \cite{CFL22}, \cite{HuI25}]
\label{local-uniqueness-Sp}
For $n\geq 2$, $\alpha\in(0,1)$ and $p\in[0,1)$, there exists a constant $\varepsilon_0>0$, that depends only on $n$, $\alpha$, $p$, such that if $f\in C^{\alpha}(S^{n-1})$ satisfies $\|f-1\|_{C^{\alpha}(S^{n-1})}<\varepsilon_0$, then the equation
$$
dS_{p,K}=fd\mathcal{H}^{n-1}
$$
has a unique solution in ${\cal K}^n_o$ in the sense of  Alexandrov, and $u=h_K|_{S^{n-1}}$ is a positive $C^{2,\alpha}$ solution in \eqref{LpMonge}.
\end{theo}

Theorem \ref{local-uniqueness-Sp} was previously established in Chen, Huang, Li, Liu \cite{CHLL20} in the even (or $o$-symmetric) case. The general case, however, appears to be much more challenging.  Ivaki \cite{Iva22} considered the stability of the solution of $L_p$ Monge-Amp\`ere equation \eqref{LpMonge} around 
the constant $1$ function from another point of view. He proved under various conditions depending on $p$ that if $p>-n$
and $f$ is close to the constant $1$ function, then any solution of \eqref{LpMonge} is close to the constant $1$ function.
 In addition, the recent paper Hu, Ivaki \cite{HuI25} proves that the unique solution of the isotropic $L_0$ Minkowski problem is stable in a very strong sense in terms of the $L_2$ distance.

Next, we discuss notions related to the $q$th dual curvature measure $\widetilde{C}_{q,K}$ on $S^{n-1}$ (cf. Section~\ref{sec-basics}) initated by Huang, Lutwak, Yang, Zhang \cite{HLYZ16}. For $q>0$ and $K\in \mathcal{K}^n_o$, the radial function $\varrho_K$  on $S^{n-1}$ is defined by
$$
\varrho_K(v)=\max\{t\geq 0:\,tv\in K\}
$$
for $v\in S^{n-1}$, and for $\omega\subset S^{n-1}$, let
$$
\alpha^*_K(\omega)=\{v\in S^{n-1}:\nu_K(\varrho_K(v)\cdot v)\cap\omega\neq \emptyset \},
$$
which is $\mathcal{H}^{n-1}$ measurable if $\omega$ is measurable.
Now for a $\mathcal{H}^{n-1}$ measurable $\omega\subset S^{n-1}$, its $q$th dual curvature measure is
$$
\widetilde{C}_{q,K}(\omega)=\int_{\alpha^*_K(\omega)}\varrho_K^{q}\,d\mathcal{H}^{n-1}.
$$
Since $\widetilde{C}_{q,K}(\{v\in S^{n-1}:\,h_K(v)=0\})=0$ for given $q>0$ and $K\in\mathcal{K}_o^n$ (cf. \eqref{Cq-nuo-zero}), it is possible to extend the definition of Lutwak, Yang, Zhang \cite{LYZ18} of the $L_p$ $q$th dual  curvature measure on $S^{n-1}$ as
$$
d\widetilde{C}_{p,q,K}= u^{-p}\,d\widetilde{C}_{q,K}
$$
for $u=h_K|_{S^{n-1}}$ and  $p\in\R$ (cf. B\"or\"oczky, Fodor \cite{BoF19}). It is readily a finite measure on $S^{n-1}$ if $o\in{\rm int}\,K$ or $p\leq 0$, and it is also a finite measure if $0<p<\min\{q,1\}$ (cf. Lemma~\ref{CpqK-bounded}). As special cases, if $q>0$ and $p<1$, then
$\widetilde{C}_{0,q,K}=\widetilde{C}_{q,K}$ and $\widetilde{C}_{p,n,K}=h_K^{1-p}\,dS_K=S_{p,K}$.
For a non-negative function $f\in L_1(S^{n-1})$, $q>0$ and $p\in\R$, the  Monge-Amp\`ere equation on $S^{n-1}$ corresponding to the $L_p$ $q$th dual Minkowski problem  is
\begin{equation}
\label{Lp-dualMink-Monge-Ampere-intro}
\left(\|\nabla u\|^2+u^2\right)^{\frac{q-n}2}\cdot u^{1-p}\det\left(\nabla^2 u+u\,I\right)=f,
\end{equation}
and a $K\in\mathcal{K}_o^n$ is a solution of the $L_p$ dual Minkowski problem \eqref{Lp-dualMink-Monge-Ampere-intro} in Alexandrov's sense (in the sense of measure) if
\begin{equation}
\label{Lp-dualMink-Monge-Ampere-measure-intro}
d\widetilde{C}_{p,q,K}=f\,d\mathcal{H}^{n-1}
\end{equation}
where we assume that $\mathcal{H}^{n-1}(\{x\in\partial'K:\langle x,\nu_K(x)\rangle=0\})=0$ when $p\geq 1$.

Huang, Zhao \cite{HuZ18} proved using the maximum principle that if $p>q$, then the solution of \eqref{Lp-dualMink-Monge-Ampere-intro} is unique.

\begin{theo}[Huang, Zhao]
\label{uniqueness-Cpq-p>q}
For $n\geq 2$, $\alpha\in(0,1)$, $p>q>0$ and positive $f\in C^{\alpha}(S^{n-1})$, the   $L_p$ $q$th dual Minkowski problem \eqref{Lp-dualMink-Monge-Ampere-intro} has a unique positive solution $u\in C^{2,\alpha}(S^{n-1})$.
\end{theo}
\noindent{\bf Remark. } If in addition, $p>\max\{q,1\}$, then even the solution $K\in\mathcal{K}_o^n$ of \eqref{Lp-dualMink-Monge-Ampere-measure-intro} in the sense of measure is unique, and hence $u=h_K|_{S^{n-1}}$ is the unique positive $C^{2,\alpha}$ solution of \eqref{Lp-dualMink-Monge-Ampere-intro} according to Remark~\ref{CpqK-bounded} (ii) and Lemma~\ref{Gauss-map-bijective}.\\

Therefore, the case $p<q$ is the interesting one concerning uniqueness in \eqref{Lp-dualMink-Monge-Ampere-intro}. According to Li, Liu, Lu \cite{LLL22}, uniqueness may not hold in \eqref{Lp-dualMink-Monge-Ampere-intro} if $p<0$ and $q=n$ possibly even if
$f$ and $u$ are positive even $C^\infty$ functions.
In this paper, we prove the following extension of Theorem~\ref{local-uniqueness-Sp}.

\begin{theo}
\label{local-uniqueness-Cpq}
For $n\geq 2$, $\alpha\in(0,1)$ and $p\in(-1,1)$, there exists a constant $\varepsilon_0\in(0,1)$ that depends only on $n$, $\alpha$ and $p$ such that if $|q-n|<\varepsilon_0$ and 
$f\in C^{\alpha}(S^{n-1})$ satisfies $\|f-1\|_{C^{\alpha}(S^{n-1})}<\varepsilon_0$, then the equation
$$
d\widetilde{C}_{p,q,K}=fd\mathcal{H}^{n-1}
$$
has a unique solution in ${\cal K}^n_o$ in the sense of  Alexandrov, and $u=h_K|_{S^{n-1}}$ is a positive $C^{2,\alpha}$ solution of \eqref{Lp-dualMink-Monge-Ampere-intro}.
\end{theo}
\noindent{\bf Remark.} Even if $q=n$ (and hence $\widetilde{C}_{p,q,K}=S_{p,K}$),  Theorem~\ref{local-uniqueness-Cpq} includes the case $p\in(-1,0)$ left open by Theorem~\ref{local-uniqueness-Sp}. If $n=2$, then the extention of Theorem~\ref{local-uniqueness-Sp} to the case $p\in(-1,0)$ has been already verified by Du \cite{Du21}. \\

We note that the $C^0$ estimate of Theorem~\ref{qcloseton-dualcurvcloseto1}~ (b) - that is a strengthened version  of the $C^0$ estimate Theorem~\ref{qcloseton-dualcurvcloseto1}~ (a) we use to prove Theorem~\ref{local-uniqueness-Cpq}, - does not hold if $-n<p\leq -1$ according to Jian, Lu, Wang \cite{JLW15}. Therefore, we propose the following question.

\begin{problem}
\label{local-uniqueness-Sp-psmall}
For $n\geq 2$, does Theorem~\ref{local-uniqueness-Sp} hold if $-n<p\leq -1$?
\end{problem}

Concerning the even case, Chen, Huang, Zhao \cite{CHZ19} proved that if $p\geq -n$ and $q\leq \min\{n,n+p\}$ with $q\neq p$, then the even solution of the isotropic $L_p$ $q$th dual Minkowski problem \eqref{Lp-dualMink-Monge-Ampere-intro} is unique.
This result has been extended to the case  when $p\geq -n$ and $q\leq n$, and one of these inequalities is strict by Ivaki, Milman \cite{IvM23}.

\begin{theo}
\label{local-uniqueness-even-Cpq}
For $n\geq 2$ and $\alpha\in(0,1)$, and for $p>-1$ and $q>0$ with $p<q<\min\{n,n+p\}$,  there exists a constant $\varepsilon_0\in(0,1)$ that depends on $n$, $\alpha$, $p$ and $q$ such that if 
$f\in C^{\alpha}(S^{n-1})$ is even and satisfies $\|f-1\|_{C^{\alpha}(S^{n-1})}<\varepsilon_0$, then the equation
$$
d\widetilde{C}_{p,q,K}=fd\mathcal{H}^{n-1}
$$
has a unique $o$-symmetric solution in ${\cal K}^n_o$ in the sense of  Alexandrov, and  $u=h_K|_{S^{n-1}}$ 
is the unique even solution of \eqref{Lp-dualMink-Monge-Ampere-intro} that 
is in addition, positive and $C^{2,\alpha}$.
\end{theo}

If $\max\{0,n-4\}<q<n$ and $p=0$ (the case of $\widetilde{C}_{q,K}=\widetilde{C}_{0,q,K}$), then J. Hu \cite{JinrongHu} proved Theorem~\ref{local-uniqueness-even-Cpq} in a stronger form; namely, J. Hu \cite{JinrongHu}  proved a stability stament in terms of the $L_2$ distance. In addition, Cabezas-Moreno, Hu \cite{CMH} prove Theorem~\ref{local-uniqueness-even-Cpq} if $1<p<q\leq n$.

We note that question of uniqueness of the solution has been discussed in various versions of the Minkowski problem.
The Gaussian surface area measure of a $K\in\mathcal{K}^n$ is defined by
Huang, Xi and Zhao \cite{HXZ21}, whose results are extended  by Feng, Liu, Xu \cite{FLX23}, Liu \cite{Liu22} and Feng, Hu, Xu \cite{FHX23}. The fact that only balls have the property that the density of the Gaussian surface area measure is constant 
is proved by Chen, Hu, Liu, Zhao \cite{CHLZ} for convex domains in $\R^2$, and in the even case for $n\geq 3$
by \cite{CHLZ} and Ivaki, Milman \cite{IvM23}. The intensively investigated $L_p$-Minkowski conjecture
stated by B\"or\"oczky, Lutwak, Yang,  Zhang \cite{BLYZ12}
claims the uniqueness of the even solution of \eqref{LpMonge} for even positive $f$,
see for example
B\"or\"oczky, Kalantzopoulos \cite{BoK22},
Chen, Huang, Li,  Liu \cite{CHLL20},
Colesanti,  Livshyts, Marsiglietti \cite{CLM17},
Colesanti,  Livshyts \cite{CoL20},
Ivaki, E. Milman \cite{IvM23b}, 
Kolesnikov \cite{Kol20},
Kolesnikov, Livshyts \cite{KoL},
Kolesnikov, Milman \cite{KoM22},
Livshyts, Marsiglietti, Nayar, Zvavitch \cite{LMNZ20},
Milman \cite{Mil24,Mil25}, 
Saroglou \cite{Sar15,Sar16},
Stancu \cite{Stancu,Stancu1,Sta22},
van Handel \cite{vHa} 
for partial results, and B\"or\"oczky \cite{Bor23} for a survey of the subject.

For the $L_p$ $q$th dual Minkowski problem
due to Lutwak, Yang, Zhang \cite{LYZ18} (see Huang, Lutwak, Yang, Zhang \cite{HLYZ16}
for the original $q$th dual Minkowski problem), Chen, Li \cite{ChL21} and
Lu, Pu \cite{LuP21} solve essentially the case $p>0$, and Huang, Zhao \cite{HuZ18} and Guang, Li, Wang \cite{LGW23} discuss the case $p<0$ (see Gardner, Hug, Weil, Xing, Ye \cite{GHWXY19,GHXY20} for Orlicz versions of some of these results using the variational method, and Li, Sheng, Wang \cite{LSW20} for an approach via the flow method). Uniqueness of the solution of the $L_p$ $q$th  dual Minkowski problem is thoroughly investigated by Li, Liu, Lu \cite{LLL22}. The case when $n=2$ and $f$ is a constant function has been completely clarified by
Li,  Wan \cite{LiW}. If the $L_p$ $q$th dual curvature is constant, then Ivaki, Milman \cite{IvM23} shows uniqueness of the even solution when $p>-n$ and $q\leq n$.\\

Let us turn to the $C^0$ estimates related to Theorem~\ref{local-uniqueness-Cpq} and 
Theorem~\ref{local-uniqueness-even-Cpq}. We only consider the case when $q$ is close to $n$, but these estimates are sufficient to prove Theorem~\ref{local-uniqueness-Cpq}.
	
\begin{theo}
\label{qcloseton-dualcurvcloseto1}
For any $n\geq 3$, $\lambda\geq 2$ and $-1<p<1$, there exist $\varepsilon_0\in(0,1)$ and $C>1$ depending on $n$, $\lambda$ and $p$ such that if $n-\varepsilon_0<q< n+\varepsilon_0$ and a convex body $K\in \mathcal{K}^n_o$ satisfies that
\begin{description}
\item{(a)} either $(1-\varepsilon_0)\mathcal{H}^{n-1}\leq \widetilde{C}_{p,q,K}\leq (1+\varepsilon_0)\mathcal{H}^{n-1}$,
\item{(b)} or $n=3,4$ and $\lambda^{-1}\mathcal{H}^{n-1}\leq \widetilde{C}_{p,q,K}\leq\lambda\mathcal{H}^{n-1}$,
\end{description}
then ${\rm diam}\,K\leq C$ and $|K|\geq C^{-1}$.
\end{theo}
\noindent{\bf Remarks.} 
\begin{itemize}
\item The case $p=0$ and $n-\varepsilon_0<q< n+\varepsilon_0$ in the results above deals with the $q$th dual curvature measure $\widetilde{C}_{q,K}$, and the case $q=n$ and $-1<p<1$  deals with the case of Lutwak's $L_p$ surface area $\widetilde{C}_{p,n,K}=S_{p,K}$. Our results are new even for $S_{p,K}$ if $-1<p<0$.

\item For any $p\in(-n,-1)$ and $n\geq 3$, Jian, Lu, Wang \cite{JLW15} (see (2.12), (2.13) and (2.14) in \cite{JLW15}) prove the existence of a $\lambda>1$ depending on $p$ and $n$ such that there exists origin symmetric convex body $K \subset\R^n$ with axial rotational symmetry and $C^\infty_+$ boundary satisfying that
$$
\frac1{\lambda}\cdot\mathcal{H}^{n-1}\leq S_{p,K}\leq \lambda\cdot \mathcal{H}^{n-1},
$$
and $|K|$ is arbitrary small; namely, for any $\varepsilon>0$ there exists such a $K$ with $|K|<\varepsilon$. This shows that the lower bound $-1$ for $p$ in the results above is optimal.

\end{itemize}

Finally, we provide the $C^0$ estimate for $o$-symmetric convex bodies that is needed for the proof of Theorem~\ref{local-uniqueness-even-Cpq}.

\begin{theo}
\label{lambda-Cpq-symmetric}
For  $\lambda>1$, $-1<p<q<\min\{n,n+p\}$ with $q>0$, 
there exists $C>1$ depending on $n$, $\lambda$, $p$ and $q$, such that if an $o$-symmetric convex body $K\subset\R^n$ satisfies that  $\lambda^{-1}\mathcal{H}^{n-1}\leq \widetilde{C}_{p,q,K}\leq\lambda\mathcal{H}^{n-1}$,
then 
$$
{\rm diam}\,K\leq C\mbox{ \ and \ }|K|\geq 1/C.
$$
\end{theo}

Concerning the structure of the paper, Section~\ref{sec-basics} summarizes the basic properties of $L_p$ dual curvature measures that we need, and Section~\ref{secRegularity} describe the correponding regularity properties.
Theorem~\ref{qcloseton-dualcurvcloseto1} is proved 
in Section~\ref{sec-qcloseton-dualcurvcloseto1}, and the proof of Theorem~\ref{lambda-Cpq-symmetric}
is prepared in Section~\ref{sec-basic-estimate-symmetric}, and completed in Section~\ref{sec-lambda-Cpq-symmetric}.
Finally, Theorems~\ref{local-uniqueness-Cpq} and \ref{local-uniqueness-even-Cpq} are proved in Section~\ref{secUniqueness}.

\section{Some basic notation and notions}
\label{sec-basics}

This section discusses the basic notions and notations used throughout the paper. For notions in convexity, see Schneider \cite{Sch14}. 
The measure  of the emptyset with respect to any measure in this paper and an integral over the emptyset is understood to be zero.  For an integer $m>0$, we frequently use the Hausdorff measure $\mathcal{H}^{m}$ on subsets of $\R^n$ where $\mathcal{H}^{m}$ is normalized in a way such that it coincides with the Lebesgue measure on $m$-dimensional linear subspaces of $\R^n$. In particular, $\mathcal{H}^{n-1}$ is the Lebesgue measure (canonical Haar measure) on $S^{n-1}$. To simplify notation, for $\mathcal{H}^{n}$-measurable $X\subset\R^n$, we frequently write 
$$
|X|=\mathcal{H}^{n}(X).
$$
The origin in $\R^n$ is denoted by $o$. As usual,  ${\rm int}\,X$ stands for the interior of an $X\subset \R^n$. 
For a compact subset $K\subset \R^n$, we write ${\rm relint}\,K$ and $\partial K$ to denote its relative interior and relative boundary with respect to the topology of the linear hull of $K$ (that is the usual notion of boundary if ${\rm int}\,K\neq \emptyset$). The support function of $K$ is
$h_K(z)=\max_{x\in K}\langle x,z\rangle$ for $z\in\R^n$, that is a convex and $1$-homogeneous function on $\R^n$ where a function $\varphi:\R^n\to\R$ is $\alpha$-homogeneous for $\alpha\in\R$ if
$\varphi(\lambda x)=\lambda^\alpha \varphi(x)$ for $\lambda>0$ and $x\in\R^n\backslash \{o\}$. 

Let $K\subset\R^n$ be a convex body; namely, $K$ is a compact convex set with ${\rm int}\,K\neq \emptyset$. In this case, $h_K$ is even twice differentiable in Alexandrov's sense at $\mathcal{H}^{n-1}$ a.e. $w\in S^{n-1}$ and $\mathcal{H}^{n}$ a.e. $z\in \R^n$, and we write  $Dh_K(z)$ to denote the Euclidean gradient (derivative) whenever it exists at a $z\in\R^n$. 
For $x\in\partial K$ and $Z\subset \partial K$, we write $\nu_K(x)$ and $\nu_K(Z)$ to denote the total spherical images; namely, the set of all unit vectors that are exterior normals at $x$ or at a point of $Z$.
	Here $\nu_K(x)$ is a compact spherically convex set for any $x\in\partial K$, and is contained in an open hemisphere. 
	We write $\partial'K$ to denote the set of boundary points with unique exterior unit normal, and for $x\in\partial' K$, abusing the notation, we also consider $\nu_K(x)$ as a vector in $S^{n-1}$
where $\mathcal{H}^{n-1}(\partial K\backslash \partial' K)=0$. 

Let us discuss the   fundamental (Borel) surface area measure $S_K$ on $S^{n-1}$ associated to a compact convex set $K\subset\R^n$ (cf. Schneider \cite{Sch14}): If $K$ is a convex body; namely, ${\rm dim}\,K=n$, then for any Borel function $\varphi:S^{n-1}\to\R$ that is either bounded or non-negative, we have
\begin{equation}
\label{SK-phi}
\int_{S^{n-1}}\varphi\,dS_K=
\int_{\partial' K}\varphi(\nu_K(x))\,d\mathcal{H}^{n-1}(x).
\end{equation}
For example, if $\partial K$ is $C^2_+$ and $u=h_K|_{S^{n-1}}$, then 
\begin{equation}
\label{SK-C2}
dS_K=\det\left(\nabla^2 u+u\,I\right)\,d\mathcal{H}^{n-1}
\end{equation}
where $\nabla u$ and  $\nabla^2 u$ are the spherical gradient and the Hessian of $u$ with respect to a moving orthonormal frame, and $I$ denotes the identity matrix of suitable size.

Next let $K\subset\R^n$ be a compact convex set with ${\rm dim}\,K\leq n-1$. If ${\rm dim}\,K\leq n-2$, then $S_K$ is the zero measure, and if ${\rm dim}\,K= n-1$, then there exists a $w\in S^{n-1}$ such that $K-z\subset w^\bot$ for any $z\in K$, and in this case, $S_K$ is concentrated onto $\{w,-w\}$ in a way such that
\begin{equation}
\label{SK-dim-n-1}
S_K(w)=S_K(-w)=\mathcal{H}^{n-1}(K).
\end{equation}

We consider the set $\mathcal{K}^n_o$ of convex bodies in $\R^n$ that contain the origin $o$. For $K\in \mathcal{K}^n_o$, the main topics of this section are the  dual curvature measure $\widetilde{C}_{q,K}$ introduced by
Huang, Lutwak, Yang, Zhang \cite{HLYZ16}  if $q\in\R$ and $o\in{\rm int}\,K$, their extension, the $L_p$ dual curvature measure $\widetilde{C}_{p,q,K}$ introduced by Lutwak, Yang, Zhang \cite{LYZ18} if $p,q\in\R$ and $o\in{\rm int}\,K$, and the versions of these notions when $q>0$ and possibly $o\in \partial K$, discussed by B\"or\"oczky, Fodor \cite{BoF19}.

	Let $K\in \mathcal{K}^n_o$. 
For $K\in \mathcal{K}^n_o$ and $p\in\R$, Lutwak's $L_p$ surface area $S_{p,K}$ is defined by $dS_{p,K}=u^{1-p}dS_K$
for $u=h_K|_{S^{n-1}}$
(see Lutwak \cite{Lut93}, B\"or\"oczky \cite{Bor23}); in particular, $S_K=S_{1,K}$ and $V_K=\frac1n\,S_{0,K}$ is the so-called cone-volume measure.

The radial function $\varrho_K$ of a  $K\in \mathcal{K}^n_o$ on $S^{n-1}$  is
$$
\varrho_K(w)=\max\{t\geq 0:\,tw\in K\}
$$
for $w\in S^{n-1}$.  It follows that if $v\in S^{n-1}$ is an exterior unit normal at $x=\varrho_K(w)\cdot w\in\partial K$, then
\begin{equation}
\label{hK-rhoK}
h_K(v)=\langle \varrho_K(w)\cdot w,v\rangle\leq \varrho_K(w),
\end{equation}
and hence if in addition, $h_K$ is differentiable at $v$, and $u=h_K|_{S^{n-1}}$,  then 
\begin{align}
\label{DhK-u-x}
Dh_K(v)=&\nabla u(v)+u(v)\cdot v=x;\\
\label{DhK-length-rhoK}
\|Dh_K(v)\|=&\sqrt{\|\nabla u\|^2+u^2}=\varrho_K(w)=\|x\|;\\
\label{DhK-rhoK}
h_K(v)=&u(v)\leq  \|Dh_K(v)\|.
\end{align}
	
For $K\in \mathcal{K}^n_o$, another fundamental measure is the cone volume measure $dV_K=\frac1n\,h_K\,dS_K$. To provide a geometric interpretation of the cone volume measure,
	we consider the Borel measure $\widetilde{V}_K$ on $\partial K$, namely, for measurable $Z\subset\partial K$, we define
	\begin{equation}
		\label{tildeVKZ-def} 
		\widetilde{V}_K(Z)=\frac1n\int_{Z\cap \partial' K}\langle x,\nu_K(x)\rangle\,dx=\frac1n\int_{Z\cap \partial' K}h_K(\nu_K(x))\,dx.
	\end{equation}
	In particular,  $\widetilde{V}_K$ is inspired by the pullback of the cone volume measure to $\partial K$; namely, for Borel $\omega\subset S^{n-1}$, combining \eqref{SK-phi} and \eqref{tildeVKZ-def} yields that
	\begin{equation}
		\label{tildeVKZ-VK}
		\widetilde{V}_K(\{x\in \partial' K: \nu_K(x)\in \omega\})=V_K(\omega),
	\end{equation}
	
	Next we provide an interpretation of $\widetilde{V}_K$ in terms of the radial function $\varrho_K$ of a $K\in \mathcal{K}^n_o$ on $S^{n-1}$. We consider the open convex cone $\Sigma_K=\{(0,\infty)x:x\in{\rm int}\,K\}$, and hence $\Sigma_K=\R^n$
	if $o\in{\rm int}K$. 
	Since
	$\{x\in\partial' K:\langle \nu_K(x),x\rangle=0\}= \partial \Sigma_K\cap \partial' K$,
	we have $\widetilde{V}_K(\partial K\cap \partial \Sigma_K)=0$ if $\partial \Sigma_K\neq \emptyset$, the paper 
	Huang, Lutwak, Yang, Zhang \cite{HLYZ16} working on the case  when $o\in{\rm int}\,K$ yields that if $Z\subset \partial K$ is measurable, then
\begin{equation}
\label{tildeVKZ-def2}
\widetilde{V}_K(Z)=\frac1n\int_{\pi_{S^{n-1}}(Z)}\varrho_K^n\,\mathcal{H}^{n-1}=\left|\bigcup\{{\rm conv}\{o,x\}:\,x\in Z\}\right|.
\end{equation} 
For any measurable $\varphi:\,\partial K\to[0,\infty)$, we deduce from \eqref{tildeVKZ-def}  that 
\begin{equation}
\label{tildeVK-varphi-rhoK}
\int_{\partial K}\varphi(x)\,d\widetilde{V}_K(x)=\int_{\partial' K}\varphi(x)\langle \nu_K(x),x\rangle\,d\mathcal{H}^{n-1}(x)
\end{equation} 
We note that if $\omega\subset S^{n-1}$, $Z=\nu_K^{-1}(\omega)\subset \partial K$ and $\Phi\in{\rm GL}(n)$ then $\langle \Phi x,\Phi^{-t}v\rangle=\langle x,v\rangle$ for $x,v\in\R^n$ yields that 
	\begin{equation}
		\label{normals-linear-image}
		\nu_{\Phi K}^{-1}\left(\Phi^{-t}_*\omega\right)=\Phi Z\mbox{ \ for \ }\Phi^{-t}_*\omega=\left\{\frac{\Phi^{-t}v}{\|\Phi^{-t}v\|}:\,v\in\omega\right\}.
	\end{equation}
	Now $\widetilde{V}_K$ nicely intertwines with linear transformations; namely, if $\Phi\in{\rm GL}(n)$ and $Z\subset\partial K$ is $\mathcal{H}^{n-1}$ measurable, then \eqref{tildeVKZ-def2} implies that
\begin{equation}
\label{tildeVK-linear-inv}
\widetilde{V}_{\Phi\,K}(\Phi\,Z)=
|\det\Phi|\cdot \widetilde{V}_K(Z),
\end{equation}
and hence for any measurable $\omega\subset S^{n-1}$, we deduce from \eqref{normals-linear-image} that
\begin{equation}
\label{VK-linear-inv}
V_{\Phi\,K}\left(\Phi^{-t}_*\omega\right)=
|\det\Phi|\cdot V_K(\omega),
\end{equation}
thus if $\varphi:S^{n-1}\to [0,\infty)$ is continuous, then
\begin{equation}
\label{VK-linear-inv-phi}
\int_{S^{n-1}}\varphi\,dV_{K}=|{\rm det} \Phi|^{-1} \int_{S^{n-1}}\varphi\circ \Phi^{t}_*\,dV_{\Phi K}.
\end{equation}
It follows from \eqref{tildeVKZ-def2} that
\begin{equation}
\label{VK-sphere-volume}
V_K(S^{n-1})=\widetilde{V}_K(\partial K)=|K|.
\end{equation}

For $q>0$, extending the notion introduced by Huang, Lutwak, Yang, Zhang \cite{HLYZ16} for $K\in \mathcal{K}^n_{(o)}$,
B\"or\"oczky, Fodor \cite{BoF19} defined
the $q$th dual curvature measure  of a $K\in \mathcal{K}^n_o$ as follows: 
For $\omega\subset S^{n-1}$, let
$$
\alpha^*_K(\omega)=\{v\in S^{n-1}:\nu_K(\varrho_K(v)\cdot v)\cap\omega\neq \emptyset \}
$$
which is $\mathcal{H}^{n-1}$ measurable if $\omega$ is measurable.
Now for $\mathcal{H}^{n-1}$ measurable $\omega\subset S^{n-1}$, its $q$th dual curvature measure is
\begin{equation} 
\label{Cq-def}
\widetilde{C}_{q,K}(\omega)=\int_{\alpha^*_K(\omega)}\varrho_K^{q}\,d\mathcal{H}^{n-1}.
\end{equation}
By the definition of $\varrho_K$, if $o\in\partial K$ and $q>0$, then 
\begin{equation}
\label{Cq-nuo-zero}
\widetilde{C}_{q,K}(\nu_K(o))=\widetilde{C}_{q,K}\Big(\{v\in S^{n-1}:\,h_K(v)=0\}\Big)=0.
\end{equation}
According to \eqref{tildeVK-varphi-rhoK}, the dual curvature measure can be written as a surface integral (cf. Huang, Lutwak, Yang, Zhang \cite{HLYZ16})
\begin{align}
\label{dual-curvature-cone-surface}
\widetilde{C}_{q,K}(\omega)=&
\int_{\nu_K^{-1}(\omega)\cap\partial' K}\|x\|^{q-n}\langle x,\nu_K(x)\rangle\,d\mathcal{H}^{n-1}(x)\\
\label{dual-curvature-cone-vol}
=&n\int_{\nu_K^{-1}(\omega)\cap\partial' K}\|x\|^{q-n}\,d\widetilde{V}_K(x)
\end{align}
for a measurable $\omega\subset S^{n-1}$ where recall that the measure of the emptyset and an integral over the emptyset is understood to be zero. It follows from \eqref{dual-curvature-cone-surface} that if $\varphi:\omega\to[0,\infty)$ is  measurable, then
\begin{align} 
\label{dual-curvature-phi-cone-on-boundary}
\int_{S^{n-1}}\varphi\,d\widetilde{C}_{q,K}=&
\int_{\partial' K}\varphi(\nu_K(x))\cdot \|x\|^{q-n}\langle x,\nu_K(x)\rangle\,d\mathcal{H}^{n-1}(x)\\
\label{dual-curvature-phi-cone-radial}
=&\int_{S^{n-1}}\varphi \cdot(\varrho_K\circ \alpha^*_K)^{q}\,d\mathcal{H}^{n-1}.
\end{align}
It follows from \eqref{dual-curvature-phi-cone-on-boundary}  that if $h_K$ is differentiable at each point of a measurable $\omega\subset S^{n-1}$ and $\varphi:\omega\to[0,\infty)$ is measurable, then
\begin{align}
\label{dual-curvature-phi-cone-surface}
\int_{\omega}\varphi\,d\widetilde{C}_{q,K}=&
\int_{\omega}\varphi \cdot \|Dh_K\|^{q-n}h_K\,dS_K\\
\label{dual-curvature-phi-cone-vol}
=& n\int_{\omega}\varphi \cdot\|Dh_K\|^{q-n}\,dV_K,
\end{align}
and hence
 \eqref{tildeVKZ-def}, \eqref{tildeVKZ-VK} and \eqref{dual-curvature-cone-vol} yield that
	$\widetilde{C}_{n,K}=nV_K$.

Let $q>0$  and $K\in\mathcal{K}_o^n$. For $p\in\R$, 
it follows from \eqref{Cq-nuo-zero} that
it is possible to extend the definition of Lutwak, Yang, Zhang \cite{LYZ18} of the $L_p$ $q$th dual  curvature measure on $S^{n-1}$ as
\begin{equation}
\label{Lpqth-dual-definition}
d\widetilde{C}_{p,q,K}= h_K^{-p}\,d\widetilde{C}_{q,K},
\end{equation}
and hence  for any measurable $\omega\subset S^{n-1}$,  \eqref{dual-curvature-phi-cone-on-boundary} and \eqref{dual-curvature-phi-cone-radial} imply that if either $p<1$, or $p\geq 1$ and $\mathcal{H}^{n-1}(\{x\in\partial'K:\langle x,\nu_K(x)\rangle=0\})=0$, then
\begin{align}
\label{lp-dual-curvature-omega-boundary}
\widetilde{C}_{p,q,K}(\omega)=&\int_{\nu_K^{-1}(\omega)\cap\partial' K} \|x\|^{q-n}\langle x,\nu_K(x)\rangle^{1-p}\,d\mathcal{H}^{n-1}(x)\\
\label{dual-curvaturep-pq-radial}
=&\int_{\alpha^*_K(\omega)}h_K\left(\nu_K(\varrho_K(v)\cdot v)\right)^{-p} \cdot\varrho_K(v)^{q}\,d\mathcal{H}^{n-1}(v),
\end{align}
which value might be infinite. We deduce from \eqref{SK-phi}, \eqref{SK-C2}, \eqref{DhK-length-rhoK} and \eqref{lp-dual-curvature-omega-boundary} that if  $q>0$, $p\in\R$ and $K\in\mathcal{K}_o^n$ has $C^2_+$ boundary, then
the  Monge-Amp\`ere equation on $S^{n-1}$ corresponding to the $L_p$ $q$th dual Minkowski problem  is
\begin{equation}
\label{Lp-dualMink-Monge-Ampere}
\left(\|\nabla u\|^2+u^2\right)^{\frac{q-n}2}\cdot u^{1-p}\det\left(\nabla^2 u+u\,I\right)=f
\end{equation}
where $u=h_K|_{S^{n-1}}$.
In addition, for a non-negative function $f\in L_1(S^{n-1})$, $q>0$ and $p\in\R$, a $K\in\mathcal{K}_o^n$ is a solution of the $L_p$ dual Minkowski problem \eqref{Lp-dualMink-Monge-Ampere} in Alexandrov's sense (in the sense of measure) if
\begin{equation}
\label{Lp-dualMink-Monge-Ampere-measure}
d\widetilde{C}_{p,q,K}=f\,d\mathcal{H}^{n-1}
\end{equation}
where we assume that $\mathcal{H}^{n-1}(\{x\in\partial'K:\langle x,\nu_K(x)\rangle=0\})=0$ when $p\geq 1$.
We note that according to B\"or\"oczky, Fodor \cite{BoF19}, if $p>1$, $q>0$ and the measure of any open hemisphere is positive, then \eqref{Lp-dualMink-Monge-Ampere-measure} always has an Alexandrov's solution satisfying that $\mathcal{H}^{n-1}(\{x\in\partial'K:\langle x,\nu_K(x)\rangle=0\})=0$.  

It is not hard the prove the following about the finiteness of the $L_p$ dual curvature measure:

\begin{remark}
\label{CpqK-bounded}
Let  $q>0$  and $K\in\mathcal{K}_o^n$. 
\begin{description}
\item{(i)}If $p<\min\{q,1\}$, or
  $o\in{\rm int}\,K$, then $\widetilde{C}_{p,q,K}(S^{n-1})<\infty$.
\item{(ii)} If $p\geq \max\{q,1\}$, $o\in\partial K$ and $\mathcal{H}^{n-1}(\{x\in\partial'K:\langle x,\nu_K(x)\rangle=0\})=0$, then $\widetilde{C}_{p,q,K}(S^{n-1})=\infty$.
\end{description}
\end{remark}

As before, let $K\in\mathcal{K}_0^n$. It follows from the definition of $\widetilde{C}_{p,q,K}$ that
$$
\widetilde{C}_{p,q,\lambda\,K}=\lambda^{q-p}\widetilde{C}_{p,q,K} \mbox{ \ for $\lambda>0$}.
$$
As special cases, if $q>0$ and $p<1$, then
\begin{align*}
\widetilde{C}_{0,q,K}=&\widetilde{C}_{q,K},\\
\widetilde{C}_{p,n,K}=&h_K^{1-p}\,dS_K=S_{p,K}\mbox{ \ is Lutwak's so-called $L_p$ surface area.}
\end{align*}
We deduce from \eqref{dual-curvature-phi-cone-surface} and \eqref{dual-curvature-phi-cone-vol} that if $h_K$ is differentiable at each point of a measurable $\omega\subset S^{n-1}$ and $\varphi:\omega\to[0,\infty)$ is measurable, then
\begin{align}
\label{lp-dual-curvature-phi-cone-surface}
\int_{\omega}\varphi\,d\widetilde{C}_{p,q,K}=&
\int_{\omega}\varphi \cdot \|Dh_K\|^{q-n}h_K^{1-p}\,dS_K\\
\label{lp-dual-curvature-phi-cone-vol}
=& n\int_{\omega}\varphi \cdot\|Dh_K\|^{q-n}h_K^{-p}\,dV_K.
\end{align}

We recall that $u=h_K|_{S^{n-1}}$, being Lipschitz, is differentiable $\mathcal{H}^{n-1}$ a.e. on $S^{n-1}$. In addition, if  $u$ is differentiable at a $w\in S^{n-1}$, then (cf. \eqref{DhK-u-x} and \eqref{DhK-length-rhoK})
$\langle x,w\rangle=h_K(w)=u(w)$ and
\begin{align*}
x=&Dh_K(w)=\nabla u(w)+u(w)\cdot w,\\
\|x\|=&\varrho_K\circ \alpha^*_K(w)=\sqrt{\|\nabla u(w)\|^2+\|u(w)\|^2}
\end{align*}
where $\nabla u$ is the spherical gradient of $u$. Thus we deduce from  \eqref{dual-curvature-phi-cone-surface} and \eqref{dual-curvature-phi-cone-vol} that if $u$ is differentiable at each point of a measurable $\omega\subset S^{n-1}$, then
	\begin{align}
		\label{dual-curvature-cone-vol1}
		\widetilde{C}_{q,K}(\omega)=&n\int_{\omega}\left(\|\nabla u\|^2+\|u\|^2\right)^{\frac{q-n}2}\cdot u\,dS_K\\
		\label{dual-curvature-cone-vol0}
		=&\int_{\omega}(\varrho_K\circ \alpha^*_K)^{q-n}\,dV_K.
	\end{align}
	Here $u$ is differentiable at each point of a measurable $\omega\subset S^{n-1}$ if and only if $\nu_K^{-1}(\omega)\subset\partial K$ contains no segment.
According to \eqref{dual-curvature-cone-vol1}, 
the  Monge-Amp\`ere equation corresponding to the $q$th dual Minkowski problem for $q>0$ is
\begin{equation}
\label{dualMink-Monge-Ampere}
\left(\|\nabla u\|^2+u^2\right)^{\frac{q-n}2}\cdot u\det\left(\nabla^2 u+u\,I\right)=f,
\end{equation}
and if $p<1$, then the $L_p$ $q$th dual Minkowski problem is
\begin{equation}
\label{Lp-dualMink-Monge-Ampere}
\left(\|\nabla u\|^2+u^2\right)^{\frac{q-n}2}\cdot u^{1-p}\det\left(\nabla^2 u+u\,I\right)=f.
\end{equation}

For compact, convex sets $K_1,K_2\subset \R^n$, their Hausdorff distance is just the $L^\infty$ norm of the restrictions of the support functions to $S^{n-1}$; namely, it is  $\min\{|h_{K_1}(x)-h_{K_2}(x)|:x\in S^{n-1}\}$.  
In this paper, convergence of compact, convex sets in $\R^n$ is always meant convergence in terms of the Hausdorff distance. 
As it was proved by Alexandrov (see for example, Schneider \cite{Sch14}), the surface area measure is weakly continuous on compact convex sets. In particular, if compact convex sets $K_m\subset \R^n$ tend to a compact convex convex set $K\subset\R^n$ in terms of the Hausdorff distance, then 
\begin{equation}
\label{convergence-support-function}
\mbox{$h_{K_m}$ tends uniformly to $h_K$},
\end{equation}
and if moreover, continuous functions $\varphi_m:S^{n-1}\to \R$ tend uniformly to a continuous function $\varphi:S^{n-1}\to \R$, then
\begin{equation}
\label{weak-convergence-SK}
\lim_{m\to \infty}\int_{S^{n-1}}\varphi_m \,dS_{K_m}=\int_{S^{n-1}}\varphi\, dS_{K}.
\end{equation}
Combining \eqref{convergence-support-function} and \eqref{weak-convergence-SK}, we deduce that if $p<1$
and $K_m\in\mathcal{K}^n_o$ tends to a $K\in\mathcal{K}^n_o$ as $m$ tends to infinity, then
\begin{equation}
\label{weak-convergence-SpK}
\mbox{$S_{p,K_m}$ tends weakly to $S_{p,K}$}.
\end{equation}
We note according the Blaschke selection theorem, any bounded sequence $K_m\in \mathcal{K}^n_o$ of convex bodies contains a subsequence that tends to a compact convex set $K\subset\R^n$ with $o\in K$.
	
\section{Regularity if the dual Minkowski curvature is bounded and bounded away from zero}
\label{secRegularity}

Our technical Lemma~\ref{Gauss-map-bijective} is based on Caffarelli's classical papers \cite{Caf90a,Caf90b,Caf93} (see 
Figalli \cite{Fig17} for a more readable account). We recall that either in $\R^d$ or in $S^d$, a boundary point $x$ is an extreme point of a convex set $C$ if a (geodesic) segment in the closure of $C$ contains $x$, then $x$ is an endpoint of the segment. We note that if $C$ is bounded, then the closure of $C$ is the convex hull of the extreme points. For $X\subset \R^d$, we set 
$$
{\rm pos}_+X=\{t x:\,x\in X\mbox{ and }t>0\}.
$$
Let us recall some fundamental properties of Monge-Amp\`ere equations based on the monograph  Figalli \cite{Fig17} and
 the survey article Trudinger, Wang~\cite{TrW08}.
Given a convex function $v$ defined in an open convex set $\Omega\subset\R^{d}$, $d\geq 2$, $D v$ and $D^2 v$ denote its gradient and its Hessian, respectively. The subgradient $\partial v(x)$ of $v$ at $x\in\Omega$ is defined as
$$
\partial v(x) =\{z\in\R^{d} : v(y)\geq v(x)+\langle z,y-x\rangle \text{ for each $y\in\Omega$}\},
$$
which is a compact convex set.
If $\omega\subset\Omega$ is a Borel set, then its associated Monge-Amp\`ere measure is 
\begin{equation}
\label{Monge-Ampere-measure}
\mu_v(\omega)=\mathcal{H}^d\Big(\bigcup_{x\in\omega}\partial v (x)\Big).
\end{equation}
We observe that if $v$ is $C^2$, then
$$
\mu_v(\omega)=\int_\omega \det( D^2 v)\,d\mathcal{H}^{d}.
$$
The following regularity properties of Monge-Amp\'ere equations have been proved in a series of papers by Caffarelli, and are thoroughly presented in a simplified manner by Figalli \cite{Fig17}.

\begin{theo}[Caffarelli]
\label{Caffarelli}
For $\lambda>1$, a bounded open convex set $\Omega\subset\R^d$ with $d\geq 1$,  and the measurable function $f$ on $\Omega$ with $1/\lambda\leq f\leq \lambda$, let the convex function $v$ on $\Omega$ satisfy
\begin{equation}
\label{Caffarelli-eq}
\det v=f
\end{equation}
in the sense of measure; namely, $d\mu_v=f\,d\mathcal{H}^d$ for the Monge-Amp\'ere measure $\mu_v$.
\begin{description}
\item{(i)} If $v(x)\geq 0$ for $x\in\Omega$, and there exists an extreme point $x_0\in\Omega$ of the convex set $\{v=0\}$, then $\{v=0\}=\{x_0\}$.
			
\item{(ii)} If $v$ is strictly convex, then $v$ is locally $C^{1,\alpha}$ for an $\alpha\in(0,1)$ depending on $d$ and $\lambda$.
			
\item{(iii)} If $f$ is locally $C^{0,\alpha}$ for an $\alpha\in(0,1)$, then $v$ is locally $C^{2,\alpha}$.

			
\end{description}
\end{theo}

The main result of this section is the following statement.

\begin{lemma}
\label{Gauss-map-bijective}
Let $q>0$, $\lambda>1$ and $K\in\mathcal{K}^n_o$, and 
\begin{itemize}
\item either let $p<1$,
\item or let $p\in\R$ and $o\in {\rm int}\,K$ 
\end{itemize}
where $d\widetilde{C}_{p,q,K}=fd\mathcal{H}^{n-1}$ for a measurable  $f:\,S^{n-1}\to [\frac1{\lambda},\lambda]$.
\begin{description}
\item[(i)] $h_K$ is $C^{1}$  on the set $\{h_K>0\}$; and $\partial K\backslash\Gamma_K$ is $C^1$ and contains no segment for 
$$
\Gamma_K=\{x\in\partial K:\,\exists w\in\nu_K(x) \mbox{ with } h_K(w)=0\}.
$$
			
\item[(ii)]  If $Z\subset\partial K$ is measurable, and $\omega=\nu_K(Z)$, then
\begin{align}
\label{Gauss-map-bijective-eq}
\widetilde{C}_{p,q,K}(\omega)=&n\int_{Z}\langle x,\nu_K(x)\rangle^{-p}\|x\|^{q-n}\,d\widetilde{V}_K(x)\\
\label{Gauss-map-bijective-SK-eq}
=&n\int_{Z\cap\partial' K}\langle x,\nu_K(x)\rangle^{1-p}\|x\|^{q-n}\,d\mathcal{H}^{n-1}(x).
\end{align}
			
\item[(iii)] $\varrho_K\circ\alpha^*_K(v)=\|Dh_K(v)\|>0$ for both $\mathcal{H}^{n-1}$ a.e. and  $\widetilde{C}_{p,q,K}$ a.e. $v\in S^{n-1}$, and we have
\begin{align}
\label{CpqK-from-cone-volume}
d\widetilde{C}_{p,q,K}=&nh_K^{-p}(\varrho_K\circ\alpha^*_K)^{q-n}\,dV_K=nh_K^{-p}\|Dh_K\|^{q-n}\,dV_K\\
\label{CpqK-from-SK}
=& h_K^{1-p}\|Dh_K\|^{q-n}\,dS_K\\
\label{cone-volume-from-CqK}
dV_K=&\frac{h_K^{p}}n(\varrho_K\circ\alpha^*_K)^{n-q}\,d\widetilde{C}_{p,q,K}=\frac{h_K^{p}}n\cdot \|Dh_K\|^{n-q}\,d\widetilde{C}_{p,q,K}.
\end{align}

\item[(iv)] If $o\in\partial K$, then $\mathcal{H}^{n-1}(\nu_K(o))=0$.

\item[(v)] If, in addition, $f$ is positive and $C^{0,\alpha}$ for an $\alpha\in(0,1)$, then $h_K$ is locally $C^{2,\alpha}$ on $\{h_K>0\}$, and $\partial K\backslash\Gamma_K$ is locally $C^2_+$.


\end{description}
		
\end{lemma}
\noindent{\bf Remark. } Even if $f$ is positive and $C^{0,\alpha}$ for an $\alpha\in(0,1)$, it is possible that $\Gamma_K\subset \partial K$ contains an $(n-1)$-ball centered at $o$ (cf. Example~4.2 in Bianchi, B\"or\"oczky, Colesanti \cite{BBC20}).\\
\proof Let $u=h_K|_{S^{n-1}}$.

First we consider any $e\in S^{n-1}$ such that $h_K(e)>0$. Now the function $v(y)=h_K(y+e)$ of $y\in e^\bot$ satisfies that 
$v(y)=(1+\|y\|^2) u(w)$ for
$w=(1+\|y\|^2)^{-1}(y+e)\in S^{n-1}$. In addition, we deduce from \eqref{dualMink-Monge-Ampere} that the Monge-Amp\`ere equation
\begin{equation}
\label{MongeAmpereRn}
\det( D^2 v(y) )=v(y)^{p-1}\left(\|Dv(y)\|^2+(\langle Dv(y),y\rangle-v(y))^2\right)^{\frac{n-q}2}\cdot g(y) 
\end{equation}
holds for $y\in e^\bot$ in a small neighborhood of $o$ where 
$$
g(y)=\left(1+\|y\|^2\right)^{-\frac{n+p}2} f\left(\frac{e+y}{\sqrt{1+\|y\|^2}}\right).
$$

For (i), first we claim that if $p_1,p_2\in S^{n-1}$ are independent and $h_K(p_i)>0$, $i=1,2$, then 
\begin{equation}
\label{hkstrictly-convex}
h_K(sp_1+tp_2)<sh_K(p_1)+t h_K(p_2)
\end{equation}
holds for $s,t>0$. We suppose that the claim does not hold, and hence $h_K(sp_1+tp_2)=sh_K(p_1)+t h_K(p_2)$, and seek a contradiction. For a $z\in\partial K$ such that $sp_1+tp_2\in \nu_K(z)$, we deduce that $p_1,p_2\in \nu_K(z)$. In particular,
$\nu_K(z)$ is an at least one dimensional spherically convex compact set, $h_K(p)=\langle z,p\rangle$ for $p\in \nu_K(z)$, and  there exists a extreme point $e$ of  $\nu_K(z)$ with $h_K(e)>0$. We observe that if $x\in \R^n$, then
\begin{equation}
\label{hK-estimate-equality-e}
h_K(x)\geq \langle z,x\rangle\mbox{ and }h_K(e)= \langle z,e\rangle.
\end{equation}
We choose a $\varrho>0$ such that $h_K(x)>\frac12\,h_K(e)$ if $\|x-e\|<2\varrho$ for $x\in\R^n$, and let $\Omega=\{y\in e^\bot:\,\|y\|< \varrho\}$. It follows that there exists $\theta>0$ depending on $K,\lambda,e,\varrho$ such that the right hand side of \eqref{MongeAmpereRn} is between $1/\theta$ and $\theta$ if $y\in \Omega$.
	Now by the choice of $e$, the set $C$ of all $y\in \Omega$ such that $\|y\|\leq \varrho$ and
	$(1+\|y\|^2)^{-1}(y+e)\in \nu_K(z)$ is an at least one-dimensional convex set such that $o$ is an extreme point. Since 
	\eqref{hK-estimate-equality-e} yields the convex function $\tilde{v}(y)=v(y)-\langle z,e+y\rangle$ satisfies that 
	$\tilde{v}(y)\geq 0$ for $y\in\Omega$,
	$\tilde{v}(y)=0$ for $y\in C$, and the Monge-Amp\`ere measure of $\tilde{v}$ is between $\frac1{\theta}\,d\mathcal{H}^{n-1}$ and $\theta\,d\mathcal{H}^{n-1}$, we have contradicted Lemma~\ref{Caffarelli} (i). In turn, we conclude \eqref{hkstrictly-convex}.
	
	We deduce from \eqref{hkstrictly-convex} that $\partial K\backslash\Gamma_K$ is $C^1$.
	
	Next, for any $e\in S^{n-1}$ with $u(e)>0$, we consider \eqref{MongeAmpereRn} for  $y\in \Omega$ in a suitably small open neighborhood $\Omega$ of $o$ such that
	$v(y)\geq \frac12\,v(o)$ for $y\in \Omega$. Now 
	the right hand side of \eqref{MongeAmpereRn} is between $1/\theta$ and $\theta$ if $y\in \Omega$ where $\theta>0$ depends on $K,\lambda,e,\Omega$. Since
	\eqref{hkstrictly-convex} yields that $v$ is strictly convex, we deduce from Lemma~\ref{Caffarelli} (ii) that after possibly shrinking $\Omega$, we may assume that $v$ is $C^{1,\beta}$ for some $\beta>0$. In turn, it follows that $h_K$ is $C^1$ on $\{h_K>0\}$, and hence 
	$\partial K\backslash\Gamma_K$ contains no segment, completing the proof of (i). 

Concerning $\nu_K(o)$ in (iv) if $o\in\partial K$, we have $0=\widetilde{C}_{p,q,K}(\nu_K(o))\geq \lambda^{-1}\mathcal{H}^{n-1}(\nu_K(o))$ by \eqref{Cq-nuo-zero} and the condition on $\widetilde{C}_{p,q,K}$, and hence $\mathcal{H}^{n-1}(\nu_K(o))=0$.
	
	If, in addition, $f$ is positive and $C^{0,\alpha}$ as in (v), then  the right hand side of \eqref{MongeAmpereRn} is actually
	$C^{0,\gamma}$ for $\gamma=\min\{\alpha,\beta\}\in(0,1)$, therefore, $v$ is  $C^{2,\gamma}$ on $\Omega$ by  Lemma~\ref{Caffarelli} (iii). It follows that the right hand side of \eqref{MongeAmpereRn} is even
	$C^{0,\alpha}$, which then yields that $v$ is locally $C^{2,\alpha}$. In turn, we conclude (v).
	
	Next, we consider (ii) without any differentiability assumption on $f$.
	The point is that if $o\in\partial K$, then ${\rm pos}_+(\partial K\backslash \Gamma_K)={\rm pos}_+{\rm int}\,K$ is an open convex cone, thus
	$$
	\int_{S^{n-1}\backslash {\rm pos}_+(\partial K\backslash \Gamma_K)} \varrho_K^q(x)\,dx=0
	$$
	by the definition of $\varrho_K$, and hence $\widetilde{C}_{n,K}(\omega)=\widetilde{C}_{n,K}(\omega_0)$ where 
	$Z_0=Z\backslash \Gamma_K$ and $\omega_0=\nu_K(Z_0)$. Since 
	$\nu_K^{-1}(\omega_0)=Z_0$ by (i),
	and $\widetilde{V}_K\left(\Gamma_K\right)=0$
	as $h_K(\nu_K(x))=0$ for $x\in\Gamma_K\cap \partial'K$, we conclude (ii) from
	\eqref{dual-curvature-cone-vol}.
	
	For (iii), we note that $\varrho_K\circ\alpha^*_K(p)$ is well-defined if $h_K$ is differentiable at $p\in S^{n-1}$, and hence 
	$\varrho_K\circ\alpha^*_K(p)$ is well-defined for $\mathcal{H}^{n-1}$ a.e. $p\in S^{n-1}$. Now if
	$\omega\subset S^{n-1}$ is measurable with
	$\mathcal{H}^{n-1}(\omega)>0$, then
$$
0<\widetilde{C}_{q,K}(\omega)=\int_{\alpha^*_K(\omega)}\varrho_K^{q}\,d\mathcal{H}^{n-1}
$$
by (iv);	therefore, $\varrho_K\circ\alpha^*_K(p)>0$ for some some $p\in\omega$. We conclude that
	$\varrho_K\circ\alpha^*_K(p)>0$ 
	for $\mathcal{H}^{n-1}$ a.e. $p\in S^{n-1}$, and hence for $\widetilde{C}_{q,K}$ a.e. $p\in S^{n-1}$, as well. Finally, we deduce \eqref{cone-volume-from-CqK} via (ii), completing the proof of
	Lemma~\ref{Gauss-map-bijective}.
	\proofbox

\section{Proof of Theorem~\ref{qcloseton-dualcurvcloseto1}}
\label{sec-qcloseton-dualcurvcloseto1}

The proof of Theorem~\ref{qcloseton-dualcurvcloseto1} will use the observations Corollary~\ref{htodelta-bound}, Lemma~\ref{kn-large-CpqS} and Lemma~\ref{scalar-top-onsphere}.
In addition, Claim~\ref{htodelta-bound0} prepares the argument for  Corollary~\ref{htodelta-bound}.

\begin{claim}
\label{htodelta-bound0}
If $R>0$, then there exists $C>0$ depending on $R$ such that
\begin{equation}
\label{h1delta-h-eq}
|h^{1+\delta}-h|\leq C|\delta|
\end{equation}
for any $\delta\in(-\frac1{4+R},\frac1{4+R})$ and $h\in[0,R]$.
\end{claim}
\proof We use that
\begin{equation}
\label{h1delta-h}
h^{1+\delta}-h=h(e^{\delta\log h}-1)
\mbox{ \ where $|e^t-1|\leq 2t$ if $|t|\leq \frac12$.}
\end{equation}

If $R>1$ and $h\in[1,R]$, then \eqref{h1delta-h-eq} follows from \eqref{h1delta-h}. Therefore, we may assume that $h\in[0,1]$.

For fixed $\delta\in (0,\frac1{4})$, derivating $f_\delta(h)=h-h^{1+\delta}\geq 0$ shows that $f_\delta$ attains its maximum at $h_\delta=(1+\delta)^{\frac{-1}{\delta}}\in [e^{-1},\frac12]$, thus \eqref{h1delta-h} yields \eqref{h1delta-h-eq}.

For fixed $\delta\in (-\frac1{4},0)$, derivating $f_\delta(h)=h^{1+\delta}-h\geq 0$ shows that $f_\delta$ attains its maximum at $h_\delta=(1+\delta)^{\frac{-1}{\delta}}\in [e,4]$, and hence again \eqref{h1delta-h} yields \eqref{h1delta-h-eq}.
\proofbox

\begin{coro}
\label{htodelta-bound}
If $p<1$ and $R>0$, then there exists a function $\varepsilon(\delta)>0$ of $\delta\in(-\frac1{4+R},\frac1{4+R})$ depending on $p$ and $R$ such that $\lim_{\delta\to 0}\varepsilon(\delta)=0$, and if $h\in[0,R]$, then
\begin{equation}
\label{h1delta-h-eq}
|h^{1-p+\delta}-h^{1-p}|\leq \varepsilon(\delta).
\end{equation}
\end{coro}

Let $\kappa_d=\mathcal{H}^d(B^d)$ for $d\in\N$.

\begin{lemma}
\label{kn-large-CpqS}
If $p<1$, $q>0$, $\eta\in(0,1)$ and $D={\rm diam}\,K$ for a convex body $K\subset \R^n$, $n\geq 1$ such that $o\in K$ and 
$\sigma_{K}+\eta\cdot D\,B^n\subset K$, then
$$
\widetilde{C}_{p,q,K}(S^{n-1})\geq \kappa_{n-1} \eta^{\max\{n-p,q-p\}}D^{q-p}
$$
\end{lemma}
\proof Let $r=\eta D$, and let $w\in S^{n-1}$ and $t\geq 0$ such that $\sigma_K=tw$, and let 
$$
Z=\left\{x\in \partial K:\, x=y+sw\mbox{ for some $y\in \sigma_K+(w^\bot\cap{\rm int}(rB^n))$ and $s>0$.}\right\}
$$
As $Z|w^\bot=w^\bot\cap{\rm int}(rB^n)$, we deduce that
$$
\mathcal{H}^{n-1}(Z)\geq \kappa_{n-1}\eta^{n-1}D^{n-1}.
$$
Now $\sigma_{K}+r\,B^n\subset K$, thus if $x\in Z\cap\partial'K$ where $y+sw$ for some $y\in \sigma_K+(w^\bot\cap{\rm int}(rB^n))$ and $s>0$, then
\begin{align*}
\langle w,\nu_K(x)\rangle= &\frac 1s\langle x-y,\nu_K(x)\rangle\geq 0,\\
\langle x,\nu_K(x)\rangle=&h_K(\nu_K(x))\geq \langle tw+r\cdot \nu_K(x),\nu_K(x)\rangle=
t\langle w,\nu_K(x)\rangle+r\\
\geq &r=\eta\cdot D,\\
\|x\|\geq &\langle x,\nu_K(x)\rangle\geq \eta\cdot D.
\end{align*}
As readily $\|x\|\leq D$, we deduce that
\begin{align*}
\|x\|^{q-n}\geq & \eta^{q-n}\cdot D^{q-n}&\mbox{if }q\geq n,\\
\|x\|^{q-n}\geq &  D^{q-n}&\mbox{if }q\leq n.
\end{align*}
Therefore, \eqref{lp-dual-curvature-omega-boundary} yields that
\begin{align*}
\widetilde{C}_{p,q,K}(S^{n-1}) \geq &
\int_{Z\cap\partial' K} \|x\|^{q-n}\langle x,\nu_K(x)\rangle^{1-p}\,d\mathcal{H}^{n-1}(x)\\
.\geq &\mathcal{H}^{n-1}(Z)\cdot \eta^{\max\{q-n,0\}}D^{q-n}(\eta\cdot D)^{1-p}
\geq \eta^{\max\{n-p,q-p\}}D^{q-p}.
\end{align*}
\mbox{ }\hfill\proofbox

\begin{lemma}
\label{scalar-top-onsphere}
If $s\in(0,1)$, then
$$
\int_{S^{n-1}\backslash e_n^\bot}\langle x,e_n\rangle^{-s}\,d\mathcal{H}^{n-1}(x)<\infty
$$
where $e_n$ is the $n$th basis vector.
\end{lemma}
\proof It follows from using polar coordinates in $B^n$ and the Fubini theorem in $[-1,1]^{n}$ that
\begin{align*}
\int_{S^{n-1}\backslash e_n^\bot}\langle x,e\rangle^{-s}\,d\mathcal{H}^{n-1}(x)=&
\frac1{n-s}\int_{B^{n}\backslash e_n^\bot}\langle x,e_n\rangle^{-s}\,dx\\
\leq& \frac1{n-s}\int_{[-1,1]^n\backslash e_n^\bot}\langle x,e_n\rangle^{-s}\,dx=\frac{2^n}{(n-s)(1-s)}.
\end{align*}
\mbox{ }\hfill\proofbox

The following Theorem~\ref{con-volume-nonexistence} is from 
Saroglou \cite{Sar} and 
B\"or\"oczky, Saroglou \cite{BoS24}. For a  linear $k$-subspace $L\subset \R^n$, $k\in\{1,\ldots,n-1\}$, and a measurable function $\psi:L\cap S^{n-1}\to[0,\infty)$,  we set
$$
{\rm ess\,inf}\,\psi:=\max_{\omega\subset L\cap S^{n-1},\;\mathcal{H}^{k-1}(\omega)=0}\inf \psi\left|_{(L\cap S^{n-1})\backslash\omega}\right. .
$$

\begin{theo}
\label{con-volume-nonexistence}
For $n\geq 3$ and $k\in\{1,\ldots,n-1\}$,  let the linear $k$-subspace $L\subset \R^n$ and $K\in\mathcal{K}^n_o$ satisfy that ${\rm supp}\,V_K\subset L\cap S^{n-1}$, and 
$V_{K}\llcorner (L\cap S^{n-1})$ is absolutely continuous with density function $\psi:L\cap S^{n-1}\to[0,\infty)$.
\begin{description}
\item{(i)} If $n=3,4$, then
\begin{equation}
\label{con-volume-nonexistence-n34}
{\rm ess}\inf \psi=0.
\end{equation}
\item{(ii)} If $n\geq 3$, then $\psi$ can't be an $\mathcal{H}^{k-1}$ a.e. constant function on $L\cap S^{n-1}$.
\end{description}
\end{theo}
\noindent{\bf Remark.} Naturally, (i) yields (ii) if $n=3,4$. However, (i) may fail, for example, if $n=5$ and $k=4$.\\

\noindent{\it Proof of Theorem~\ref{qcloseton-dualcurvcloseto1} when $n\geq 3$ and $\lambda_m$ tends to $1$.}
	We suppose that Theorem~\ref{qcloseton-dualcurvcloseto1} does not hold for some $p\in(-1,1)$, and seek a contradiction. 
It follows that  for some sequence $\lambda_m\in(1,2)$ with $\lim_{m\to\infty}\lambda_m=1$, there exist $q_m\in (n-1,n+1]$ with $\lim_{m\to\infty}q_m=n$, and $K_m\in\mathcal{K}_o^n$ such that 
\begin{align}
\nonumber
d\widetilde{C}_{p,q_m,K_m}=&f_m\,d\mathcal{H}^{n-1}\\
\label{control-shape-lambda-theoproof0}
\frac1{\lambda_m}\leq& f_m\leq \lambda_m\mbox{ for a measurable function }f_m:S^{n-1}\to\R,
\end{align}
and
\begin{description}
\item{(a)} either $\lim_{m\to\infty}{\rm diam}\,K_m=\infty$,
\item{(b)} or $\lim_{m\to\infty}|K_m|=0$.
\end{description}
We may assume that either $q_m\leq n$ for all $m$, or $q_m\geq n$ for all $m$. 

First, we assume that $\lim_{m\to\infty}{\rm diam}\,K_m=\infty$.	
Let $E_{K_m}$ be the origin centered KLS ellipsoid of $K_m-\sigma_{K_m}$ (cf. Kannan, Lovász, Simonovits \cite{KLS95}), and hence
$$
\sigma_{K_m}+E_{K_m}\subset K_m\subset \sigma_{K_m}+nE_{K_m}
$$
for the centroid $\sigma_{K_m}$ of $K_m$,
 and let $r_{1,m}\leq\ldots\leq r_{n,m}$ be the half axes of $E_{K_m}$.
As ${\rm diam}\,K\leq 2nr_{n,m}$, we have 
 	$\lim_{m\to\infty}r_{n,m}= \infty$. Therefore, we may assume that for some $k=\{1,\ldots,n\}$ and $\theta\in(0,1)$, we have
\begin{equation}
\label{control-shape-rkn-nuKo0}
\begin{array}{rcl}
r_{n-k+1,m}&\geq& \theta\cdot r_{n,m},\\[1ex]
{\displaystyle\lim_{m\to\infty}}\frac{r_{n-k,m}}{r_{n,m}}&=&0
\end{array}
\end{equation} 
where the second condition in \eqref{control-shape-rkn-nuKo0} is void if $k=n$.
As the $L_p$  dual curvature measure is invariant under orthogonal transformations, we may also assume that
 there exists a fixed orthonormal basis $e_1,\ldots,e_n$ of $\R^n$ such that $r_{i,m}e_i\in\partial K_m$ for $i=1,\ldots,n$.

As $\widetilde{C}_{p,q_m,K_m}(S^{n-1})$ is bounded by \eqref{control-shape-lambda-theoproof0}, Lemma~\ref{kn-large-CpqS} yields that
\begin{equation}
\label{k-atmost-n-1}
k\leq n-1.
\end{equation}

We observe that the diagonal linear maps $\Psi_m\in{\rm GL}(n)$ with
\begin{align*}
\Psi_m(e_i)=&\frac{1}{r_{i,m}}\cdot e_i\mbox{ \ if }i=1,\ldots,n-k,\\
\Psi_m(e_i)=&\frac{1}{r_{n,m}}\cdot e_i\mbox{ \ if }i= n-k+1,\ldots,n
\end{align*}
satisfy that $\Psi_m\sigma_{K_m}+\theta B^n\subset \Psi_mK_m\subset 2nB^n$, and hence 
 we may assume that
$$
\lim_{m\to \infty} \Psi_mK_m=K_\infty
$$
for a convex body $K_\infty\in\mathcal{K}^n_o$.
We define $\alpha_m>0$ by
$$
\det \Psi_m=\alpha_mr_{n,m}^{q_m-n-p},
$$
 thus we may also assume that
\begin{equation}
\label{alphamalpha}
\lim_{m\to\infty}\alpha_m=\alpha\in[0,\infty].
\end{equation}

We observe that 
$\frac{1}{r_{n,m}}\, \Psi_m^{-1}$ is a diagonal matrix with entries
\begin{align*}
\lambda_{i,m}=&\frac{r_{i,m}}{r_{n,m}}\mbox{ \ if }i=1,\ldots,n-k,\\
\lambda_{i,m}=&1\mbox{ \ if }i= n-k+1,\ldots,n
\end{align*}
where $\lim_{m\to \infty} \lambda_{i,m}=0$ if $i=1,\ldots,n-k$. 

To simplify the formulas, let $d_m=q_m-n$.  Now Lemma~\ref{Gauss-map-bijective} (i) says that  $h_{K_m}$ is $C^{1}$  on the set $\{h_{K_m}>0\}$, and hence $h_{\Psi_mK_m}$ is $C^{1}$  on the set $\{h_{\Psi_mK_m}>0\}$. For an open $\Omega\subset\R^n$ and $C^1$ function $g:\Omega\to \R$, we write 
$D_ig$ to denote $i$th partial derivative with respect to the basis $e_1,\ldots,e_n$, and hence
$Dg=(D_1g,\ldots,D_ng)$ is the Euclidean gradient. 

For a continuous function $\varphi:S^{n-1}\to S^{n-1}$, we claim that
\begin{align}
\nonumber
\int_{S^{n-1}}\varphi\,dS_{p,K_\infty}=&\lim_{m\to \infty} n\int_{S^{n-1}}\left(d_m^2h_{\Psi_mK_m}^2+\sum_{i=1}^n\lambda_{i,m}^2(D_ih_{\Psi_mK_m})^2\right)^{\frac{d_m}2}\cdot\\
\label{phiVKinfty-limitK-VPsiKm}
&\mbox{ \ \ \ \ \ \ \ \ \ \ \ \ \ \ }\cdot\varphi\cdot h_{\Psi_mK_m}^{-p}\,dV_{\Psi_m K_m}\\
\nonumber
=&\lim_{m\to \infty} \int_{S^{n-1}}\left(d_m^2h_{\Psi_mK_m}^2+\sum_{i=1}^n\lambda_{i,m}^2(D_ih_{\Psi_mK_m})^2\right)^{\frac{d_m}2}\cdot\\
\label{phiVKinfty-limitK-SPsiKm}
&\mbox{ \ \ \ \ \ \ \ \ \ \ \ \ \ \ \ \ }\cdot\varphi\cdot h_{\Psi_mK_m}^{1-p}\,dS_{\Psi_m K_m}.
\end{align}
Here \eqref{phiVKinfty-limitK-VPsiKm} and \eqref{phiVKinfty-limitK-SPsiKm} are equivalent by definition. 

In order to handle the case $d_m=0$ in the proof of \eqref{phiVKinfty-limitK-SPsiKm}, we use the convention $0^0=1$.
Choosing some $R>0$ such that $K_\infty\subset {\rm int}\,RB^n$, we deduce from \eqref{DhK-length-rhoK}, \eqref{DhK-rhoK}, $d_m\to 0$ and 
$\lambda_{i,m}\leq 1$ that if $m$ is large, then 
$$
d_m^2h_{\Psi_mK_m}^2\leq d_m^2h_{\Psi_mK_m}^2+\sum_{i=1}^n\lambda_{i,m}^2(D_ih_{\Psi_mK_m})^2\leq R^2
$$
$\mathcal{H}^2$ a.e. on $S^{n-1}$, and hence
$$
Q_m=h_{\Psi_mK_m}^{1-p}\left(d_m^2h_{\Psi_mK_m}^2+\sum_{i=1}^n\lambda_{i,m}(D_ih_{\Psi_mK_m})^2\right)^{\frac{d_m}2}
$$
satisfies that 
$$
\begin{array}{rclcll}
d_m^{d_m}h_{\Psi_mK_m}^{1+d_m-p}&\leq &Q_m &\leq & R^{d_m}h_{\Psi_mK_m}^{1-p}&\mbox{ if }d_m\geq 0\\[1ex]
R^{d_m}h_{\Psi_mK_m}^{1-p}&\leq &Q_m &\leq & |d_m|^{d_m}h_{\Psi_mK_m}^{1+d_m-p}&\mbox{ if }d_m\leq 0.
\end{array}
$$
As $h_{\Psi_mK_m}$ tends uniformly to $h_{K_\infty}$, Corollary~\ref{htodelta-bound} implies that $Q_m$ tends uniformly to $h_{K_\infty}^{1-p}$. Now $S_{\Psi_m K_m}$ tends weakly to $S_{K_\infty}$ (cf. \eqref{weak-convergence-SK}), proving \eqref{phiVKinfty-limitK-SPsiKm} by \eqref{weak-convergence-SK}. In turn, we conclude the claim \eqref{phiVKinfty-limitK-VPsiKm}.

Applying \eqref{VK-linear-inv-phi} with the diagonal matrix $\Phi=\Psi_m^{-1}$,  \eqref{phiVKinfty-limitK-VPsiKm} implies that
\begin{align*}
&\int_{S^{n-1}}\varphi\,dS_{p,K_\infty}&\\
=&\lim_{m\to \infty} n\det \Psi_m\int_{S^{n-1}}&\left(d_m^2h_{\Psi_mK_m}\left(\frac{\Psi_m^{-1}x}{\|\Psi_m^{-1}x\|}\right)^2+\sum_{i=1}^n\lambda_{i,m}^2D_ih_{\Psi_mK_m}\left(\frac{\Psi_m^{-1}x}{\|\Psi_m^{-1}x\|}\right)^2\right)^{\frac{d_m}2}\\
&&\cdot\varphi\left(\frac{\Psi_m^{-1}x}{\|\Psi_m^{-1}x\|}\right)\cdot h_{\Psi_mK_m}\left(\frac{\Psi_m^{-1}x}{\|\Psi_m^{-1}x\|}\right)^{-p}\,dV_{K_m}(x).
\end{align*}
Here if $h_{K_m}$ is differentiable at a $x\in S^{n-1}$, then
\begin{align*}
h_{\Psi_mK_m}\left(\frac{\Psi_m^{-1}x}{\|\Psi_m^{-1}x\|}\right)=&
\|\Psi_m^{-1}x\|^{-1}h_{K_m}(x),\\
\sum_{i=1}^n\lambda_{i,m}^2D_ih_{\Psi_mK_m}\left(\frac{\Psi_m^{-1}x}{\|\Psi_m^{-1}x\|}\right)^2=&
r_{n,m}^{-2}\left\|\Psi_m^{-1}Dh_{\Psi_mK_m}\left(\frac{\Psi_m^{-1}x}{\|\Psi_m^{-1}x\|}\right)\right\|^2\\
=&r_{n,m}^{-2}\|Dh_{K_m}(x)\|^2,\\
\det \Psi_m=&\alpha_mr_{n,m}^{d_m-p}.
\end{align*}
It follows from \eqref{cone-volume-from-CqK} and \eqref{control-shape-lambda-theoproof0} that
\begin{align}
\label{control-shape-VKmCpqm-lambdam}
dV_{K_m}=&\frac1n\cdot h_{K_m}^{p}\|Dh_{K_m}\|^{-d_m}f_m\,d\mathcal{H}^{n-1}\mbox{ \ where}\\
\label{control-shape-VKmCpqm-lambdam-fm}
\frac1{\lambda_m}\leq &f_m\leq \lambda_m.
\end{align}
We consider the linear $k$-subspace
$$
L={\rm lin}\{e_{n-k+1},\ldots,e_n\},
$$
and observe that $\Phi_m=r_{n,m}^{-1}\Psi_m^{-1}$ satisfies that
\begin{align}
\label{Phim-estimate}
|\langle x,e_n\rangle| \leq& \|\Phi_m x\|\leq 1\mbox{ \ for $x\in S^{n-1}$},\\
\label{Phim-limit}
\lim_{m\to \infty}\Phi_m x=&x|L\mbox{ \ for $x\in S^{n-1}$},\\
\label{Psim-inverse=limit}
\lim_{m\to \infty}\frac{\Psi_m^{-1}x}{\|\Psi_m^{-1}x\|}=&\frac{x|L}{\|x|L\|}\mbox{ \ for $x\in S^{n-1}\backslash L^\bot$}.
\end{align}
We deduce that (cf. \eqref{alphamalpha} for the definition of $\alpha_m$)
\begin{align*}
&\int_{S^{n-1}}\varphi\,dS_{p,K_\infty}&\\
=&\lim_{m\to \infty}n\alpha_m\int_{S^{n-1}}&\left(\frac{r_{n,m}^{2}}{\|\Psi_m^{-1}x\|^2}\cdot d_m^2\cdot h_{K_m}(x)^2+\|Dh_{K_m}(x)\|^2\right)^{\frac{d_m}2}\cdot\\
&&\cdot\varphi\left(\frac{\Psi_m^{-1}x}{\|\Psi_m^{-1}x\|}\right)\cdot 
\frac{\|\Psi_m^{-1}x\|^p}{r_{n,m}^{p}}\cdot h_{K_m}(x)^{-p}\,dV_{K_m}(x)\\
=&\lim_{m\to \infty}\alpha_m\int_{S^{n-1}}&\frac{\left(\|\Phi_mx\|^{-2}\cdot d_m^2\cdot h_{K_m}(x)^2+\|Dh_{K_m}(x)\|^2\right)^{\frac{d_m}2}}{\|Dh_{K_m}(x)\|^{d_m}}\\
&&\cdot\varphi\left(\frac{\Psi_m^{-1}x}{\|\Psi_m^{-1}x\|}\right)\cdot 
\|\Phi_mx\|^{p}\cdot f_m(x)\,d\mathcal{H}^{n-1}(x).
\end{align*}
We fix a $\tilde{p}\in(|p|,1)$, and hence we may assume that $m$ is large enough to ensure that 
$$
|d_m|\leq \tilde{p}-|p|.
$$

According to Lemma~\ref{Gauss-map-bijective} (iv), if $o\in\partial K_m$, then $\mathcal{H}^{n-1}(\nu_{K_m}(o))=0$, and hence there exists a measurable set $\Theta\subset S^{n-1}$ such that  $\mathcal{H}^{n-1}(\Theta)=0$ and
$\|Dh_{K_m}(x)\|\geq h_{K_m}(x)>0$ for $x\in S^{n-1}\backslash \Theta$. In particular, if
 $x\in S^{n-1}\backslash \Theta$, then we set
$$
A_m=\left(1+\frac{\|\Phi_mx\|^{-2}\cdot d_m^2\cdot h_{K_m}(x)^2}{\|Dh_{K_m}(x)\|^2}\right)^{\frac{d_m}2},
$$
and hence $h_{K_m}(x)\leq \|Dh_{K_m}(x)\|$ (cf. \eqref{DhK-rhoK}) and \eqref{Phim-estimate} yield that
\begin{equation}
\label{intphi-Am-estimates}
\left(1+d_m^2\|\Phi_mx\|^{-2}\right)^{\frac{-|d_m|}2}\leq A_m\leq
\left(1+d_m^2\|\Phi_mx\|^{-2}\right)^{\frac{|d_m|}2}\leq 2|\langle x,e_n\rangle|^{-|d_m|}.
\end{equation}
In addition, we have
\begin{align}
\nonumber
&\int_{S^{n-1}}\varphi\,dS_{p,K_\infty}\\
\label{intphiSp-Am}
=&\lim_{m\to \infty}\alpha_m\int_{S^{n-1}}A_m(x)
\cdot\varphi\left(\frac{\Psi_m^{-1}x}{\|\Psi_m^{-1}x\|}\right)\cdot 
\|\Phi_mx\|^{p}\cdot f_m(x)\,d\mathcal{H}^{n-1}(x).
\end{align}
Now $f_m\leq 2$, and there exists some $M>0$ such that $|\varphi|\leq M$, thus \eqref{Phim-estimate} and \eqref{intphi-Am-estimates} imply that if 
$x\in S^{n-1}\backslash \Theta$, then
$$
\left|A_m(x)
\cdot\varphi\left(\frac{\Psi_m^{-1}x}{\|\Psi_m^{-1}x\|}\right)\cdot 
\|\Phi_mx\|^{p}\cdot f_m(x)\right|\leq 4M|\langle x,e_n\rangle|^{-|p|-|d_m|}\leq 
4M|\langle x,e_n\rangle|^{-\tilde{p}};
$$
therefore, we may apply Lebesgue Dominated Convergence Theorem in \eqref{intphiSp-Am} by Lemma~\ref{scalar-top-onsphere}. Since for fixed any 
$x\in S^{n-1}\backslash \Theta$, \eqref{Phim-limit} and \eqref{intphi-Am-estimates} yield that
$$
\lim_{m\to \infty}\left(A_m(x)
\cdot\varphi\left(\frac{\Psi_m^{-1}x}{\|\Psi_m^{-1}x\|}\right)\cdot 
\|\Phi_mx\|^{p}\cdot f_m(x)\right)=
\varphi\left(\frac{x|L}{\|x|L\|}\right)\cdot \|x|L\|^p,
$$
and $S^{n-1}\cap L^\bot\subset \Theta$, we deduce that 
\begin{equation}
\label{intphi-limit}
\int_{S^{n-1}}\varphi\,dS_{p,K_\infty}=\alpha\int_{S^{n-1}}\varphi\left(\frac{x|L}{\|x|L\|}\right)\cdot \|x|L\|^p
\,d\mathcal{H}^{n-1}(x)
\end{equation}
where $0<\alpha<\infty$ follows from the fact that \eqref{intphi-limit} holds, for example, when $\varphi\equiv 1$.
We conclude from \eqref{intphi-limit} and using polar coordinates in $L$ and the coarea formula that if $\varphi:S^{n-1}\to\R$ is any continuous function, then
\begin{align}
\nonumber
&\int_{S^{n-1}}\varphi\,dS_{p,K_\infty}\\
\nonumber
=&(n-k)\kappa_{n-k}\alpha\int_{S^{n-1}\cap L}\int_0^1\varphi(z)\cdot r^{k-1+p}(1-r^2)^{\frac{n-k-2}2}
\,drd\mathcal{H}^{k-1}(z)\\
\label{phiSp-phihk-1}
=&\beta\int_{S^{n-1}\cap L}\varphi\,d\mathcal{H}^{k-1}
\end{align}
where $\beta\in(0,\infty)$ because $k-1+p\geq p>-1$, and $\beta$ is a constant depending on $\alpha,n,k$, and hence is independent of $\varphi$. In turn, 
\eqref{phiSp-phihk-1} yields that $S_{p,K_\infty}$ is concentrated onto $S^{n-1}\cap L$, and the restriction to $S^{n-1}\cap L$ is $\beta$ times the $(k-1)$-dimensional Lebesgue measure. This fact contradicts Theorem~\ref{con-volume-nonexistence} (ii); therefore, there exists $\Gamma>1$ depending on $n$, $p$ and $q$ that ${\rm diam}\,K_m\leq \Gamma$, and hence
we have contradicted the possibility $\lim_{m\to\infty}{\rm diam}\,K_m=\infty$ in (a). We deduce from the indirect hypothesis on Theorem~\ref{qcloseton-dualcurvcloseto1} that condition (b) stating the formula
\begin{equation}
\label{Theorem1-volume}
\lim_{m\to\infty}|K_m|=0
\end{equation}
holds.
As ${\rm diam}\,K_m\leq \Gamma$, we may assume that $K_m$ that tends to a lower dimensional compact convex set $C\subset\R^n$ with $o\in C$. Let $w\in S^{n-1}$ such that $C\subset w^\bot$, and let $\tilde{p}\in (0,1)$ such that $\tilde{p}>p$. We may assume that $|q_m-n|\leq \tilde{p}-p$ for each $m$.

According to \eqref{DhK-rhoK}, if $h_{K_m}(v)>0$ for $v\in S^{n-1}$, then $\left(\frac{\|Dh_K(v)\|}{h_{K_m}(v)}\right)^{\tilde{p}-p}\geq 1$. We apply this observation to formula 
\eqref{CpqK-from-SK} in Lemma~\ref{Gauss-map-bijective}, and deduce that (cf. also \eqref{control-shape-lambda-theoproof0})
\begin{align*}
0<&\limsup_{m\to\infty}\widetilde{C}_{p,q_m,K_m}(S^{n-1})=\limsup_{m\to\infty}\int_{S^{n-1}}h_{K_m}^{1-p}\|Dh_{K_m}\|^{q_m-n}\,dS_{K_m}\\
\leq&\limsup_{m\to\infty}\int_{S^{n-1}}h_{K_m}^{1-\tilde{p}}\|Dh_{K_m}\|^{q_m-n+\tilde{p}-p}\,dS_{K_m}.
\end{align*}
Here $0\leq q_m-n+\tilde{p}-p\leq 2(\tilde{p}-p)$ and $\|Dh_{K_m}\|\leq \Gamma$ for any $v\in S^{n-1}$ where $h_{K_m}$ is differentiable, thus $\|Dh_{K_m}\|^{q_m-n+\tilde{p}-p}\leq \Gamma^{2(\tilde{p}-p)}$. We deduce that
$$
0<\limsup_{m\to\infty}\int_{S^{n-1}}\Gamma^{2(\tilde{p}-p)}h_{K_m}^{1-\tilde{p}}\,dS_{K_m}.
$$
Here $S_{K_m}$ tends weakly to $S_C$ (cf. \eqref{weak-convergence-SK}), and $h_{K_m}$ tends uniformly to $h_C$. Since $t\to t^{1-\tilde{p}}$ is uniformly continuous on $[0,\Gamma]$, it also follows 
$h_{K_m}^{1-\tilde{p}}$ tends uniformly to $h_{C}^{1-\tilde{p}}$; therefore,
\begin{equation}
\label{Theorem1-volume-Ccontradiction}
0<\limsup_{m\to\infty}\int_{S^{n-1}}h_{K_m}^{1-\tilde{p}}\,dS_{K_m}=\int_{S^{n-1}}h_{C}^{1-\tilde{p}}\,dS_{C}.
\end{equation}
Since $C\subset w^\bot$, we deduce that ${\rm supp}\,S_C\subset \{-w,w\}$ (cf. \eqref{SK-dim-n-1}) and $h_{C}(\pm w)^{1-\tilde{p}}=0$. This contradicts  \eqref{Theorem1-volume-Ccontradiction}, and finally proves  
Theorem~\ref{qcloseton-dualcurvcloseto1}.
\proofbox

\noindent{\it Proof of Theorem~\ref{qcloseton-dualcurvcloseto1} when $n=3,4$ and $\lambda_m=\lambda>1$.}
We suppose that the statement about the boundedness of the diameter in Theorem~\ref{qcloseton-dualcurvcloseto1}
does not hold for some $\lambda>1$, and seek a contradiction. We mostly copy the proof of  Theorem~\ref{qcloseton-dualcurvcloseto1} in the case above (when $\lambda_m$ tends to $1$).
Now there exists a sequence $q_m\in (n-1,n+1]$ with $\lim_{m\to\infty}q_m=n$, and a sequence $K_m\in\mathcal{K}_o^n$ such that $\lim_{m\to\infty}{\rm diam}\,K_m=\infty$ and
\begin{align}
\nonumber
d\widetilde{C}_{p,q_m,K_m}=&f_m\,d\mathcal{H}^{n-1}\\
\label{control-shape-lambda-theoprooflambda}
\frac1{\lambda}\leq& f_m\leq \lambda\mbox{ for a measurable function }f_m:S^{n-1}\to\R.
\end{align}
We may assume that either $q_m\leq n$ for all $m$, or $q_m\geq n$ for all $m$. 

Again, let $E_{K_m}$ be the origin centered KLS ellipsoid of $K_m-\sigma_{K_m}$ (cf. Kannan, Lovász, Simonovits \cite{KLS95}), and hence
$$
\sigma_{K_m}+E_{K_m}\subset K_m\subset \sigma_{K_m}+nE_{K_m}
$$
where $\sigma_{K_m}$ is the centroid of $K_m$,
 and let $r_{1,m}\leq\ldots\leq r_{n,m}$ be the half axes of $E_{K_m}$.
As ${\rm diam}\,K\leq 2nr_n$, we have 
 	$\lim_{m\to\infty}r_{n,m}= \infty$. Therefore, we may assume that for some $k=\{1,\ldots,n\}$ and $\theta\in(0,1)$, we have
$$
\begin{array}{rcl}
r_{n-k+1,m}&\geq& \theta\cdot r_{n,m},\\[1ex]
{\displaystyle\lim_{m\to\infty}}\frac{r_{n-k,m}}{r_{n,m}}&=&0
\end{array}
$$
where the second condition  is void if $k=n$.
As the $L_p$  dual curvature measure is invariant under orthogonal transformations, we may also assume that
 there exists a fixed orthonormal basis $e_1,\ldots,e_n$ of $\R^n$ such that $r_{i,m}e_i\in\partial K_m$.
Since $\widetilde{C}_{p,q_m,K_m}(S^{n-1})$ is bounded by \eqref{control-shape-lambda-theoprooflambda}, Lemma~\ref{kn-large-CpqS} yields that
$$
k\leq n-1.
$$
We define $\Theta\subset S^{n-1}$, $\Psi_m\in{\rm GL}(n)$ and  $\lambda_{i,m}$ as above, and again, we may assume that
$$
\lim_{m\to \infty} \Psi_mK_m=K_\infty
$$
for a convex body $K_\infty\in\mathcal{K}^n_o$.
We define $\alpha_m>0$ by
$$
\det \Psi_m=\alpha_mr_{n,m}^{d_m-p}
$$
where $d_m=q_m-n$, thus we may also assume that
$$
\lim_{m\to\infty}\alpha_m=\alpha\in[0,\infty].
$$
Instead of \eqref{control-shape-VKmCpqm-lambdam} and \eqref{control-shape-VKmCpqm-lambdam-fm}, we have
\begin{align}
\label{control-shape-VKmCpqm-lambda}
dV_{K_m}=&\frac1n\cdot h_{K_m}^{p}\|Dh_{K_m}\|^{-d_m}f_m\,d\mathcal{H}^{n-1}\mbox{ \ where}\\
\label{control-shape-VKmCpqm-lambda-fm}
\frac1{\lambda}\leq &f_m\leq \lambda.
\end{align}
Analogously to the argument leading to \eqref{intphiSp-Am}, we deduce that
if $\varphi:S^{n-1}\to [0,\infty)$ is continuous, then
\begin{align}
\nonumber
&\int_{S^{n-1}}\varphi\,dS_{p,K_\infty}\\
\label{intphiSp-Am-lambda}
=&\lim_{m\to \infty}\alpha_m\int_{S^{n-1}}A_m(x)
\cdot\varphi\left(\frac{\Psi_m^{-1}x}{\|\Psi_m^{-1}x\|}\right)\cdot 
\|\Phi_mx\|^{p}\cdot f_m(x)\,d\mathcal{H}^{n-1}(x).
\end{align}
where if $x\in S^{n-1}\backslash\Theta$, then
\begin{align}
\label{Phim-estimate-lambda}
|\langle x,e_n\rangle|\leq &\|\Phi_mx\|\leq 1,\\
\label{Phim-limit-lambda1}
\lim_{m\to \infty}\Phi_m x=&x|L\mbox{ \ for $x\in S^{n-1}$},\\
\label{Phim-limit-lambda2}
\lim_{m\to \infty}\frac{\Psi_m^{-1}x}{\|\Psi_m^{-1}x\|}=&\frac{x|L}{\|x|L\|}\mbox{ \ for $x\in S^{n-1}\backslash L^\bot$},\\
\label{intphi-Am-estimates-lambda}
\left(1+d_m^2\|\Phi_mx\|^{-2}\right)^{\frac{-|d_m|}2}\leq& A_m\leq
\left(1+d_m^2\|\Phi_mx\|^{-2}\right)^{\frac{|d_m|}2}\leq 2|\langle x,e_n\rangle|^{-|d_m|}.
\end{align}
We fix a $\tilde{p}\in(|p|,1)$, and hence we may assume that $m$ is large enough to ensure that 
$|d_m|\leq \tilde{p}-|p|$. Now $f_m\leq \lambda$, and there exists some $M>0$ such that $\varphi\leq M$, thus \eqref{Phim-estimate-lambda} and \eqref{intphi-Am-estimates-lambda} imply that if 
$x\in S^{n-1}\backslash \Theta$, then
$$
\left|A_m(x)
\cdot\varphi\left(\frac{\Psi_m^{-1}x}{\|\Psi_m^{-1}x\|}\right)\cdot 
\|\Phi_mx\|^{p}\cdot f_m(x)\right|\leq  
2M\lambda|\langle x,e_n\rangle|^{-\tilde{p}};
$$
therefore, we may apply Lebesgue Dominated Convergence Theorem in \eqref{intphiSp-Am-lambda} by Lemma~\ref{scalar-top-onsphere}. Since for fixed any 
$x\in S^{n-1}\backslash \Theta$,  the formulas \eqref{Phim-limit-lambda1}, \eqref{Phim-limit-lambda2} and \eqref{intphi-Am-estimates-lambda} yield that
$$
\lim_{m\to \infty}\left(A_m(x)
\cdot\varphi\left(\frac{\Psi_m^{-1}x}{\|\Psi_m^{-1}x\|}\right)\cdot 
\|\Phi_mx\|^{p}\right)=
\varphi\left(\frac{x|L}{\|x|L\|}\right)\cdot \|x|L\|^p,
$$
and $S^{n-1}\cap L^\bot\subset \Theta$, we deduce from $f_m\leq \lambda$  that
\begin{equation}
\label{intphi-upper-lambda}
\int_{S^{n-1}}\varphi\,dS_{p,K_\infty}\leq \alpha\lambda\int_{S^{n-1}}\varphi\left(\frac{x|L}{\|x|L\|}\right)\cdot \|x|L\|^p
\,d\mathcal{H}^{n-1}(x),
\end{equation}
and from $f_m\geq 1/\lambda$  that
\begin{equation}
\label{intphi-lower-lambda}
\int_{S^{n-1}}\varphi\,dS_{p,K_\infty}\geq \frac{\alpha}{\lambda}\int_{S^{n-1}}\varphi\left(\frac{x|L}{\|x|L\|}\right)\cdot \|x|L\|^p
\,d\mathcal{H}^{n-1}(x).
\end{equation}
Applying \eqref{intphi-upper-lambda} and  \eqref{intphi-lower-lambda} to $\varphi\equiv 1$ and $p>-1$
show that $0<\alpha<\infty$.
As above, we also deduce from \eqref{intphi-upper-lambda} and  \eqref{intphi-lower-lambda} the existence of a $\beta>0$ - that is a constant depending on $\alpha,n,k$, - such that if $\varphi:S^{n-1}\to[0,\infty)$ is any continuous function, then
\begin{align}
\label{phiSp-phihk-1-upper} 
\int_{S^{n-1}}\varphi\,dS_{p,K_\infty}\leq &\beta\lambda \int_{S^{n-1}\cap L}\varphi\,d\mathcal{H}^{k-1},\\
\label{phiSp-phihk-1-lower} 
\int_{S^{n-1}}\varphi\,dS_{p,K_\infty}\geq &\frac{\beta}{\lambda} \int_{S^{n-1}\cap L}\varphi\,d\mathcal{H}^{k-1}.
\end{align}
In turn, 
\eqref{phiSp-phihk-1-upper} yields that $S_{p,K_\infty}$ is concentrated onto $S^{n-1}\cap L$, and there exists
bounded measurable $\psi:\,S^{n-1}\cap L\mapsto [0,\infty)$ such that for any Borel set $\omega\subset S^{n-1}$, we have
$$
S_{p,K_\infty}(\omega)=\int_{\omega\cap L}\psi\,d\mathcal{H}^{k-1}.
$$
Therefore, \eqref{phiSp-phihk-1-lower} contradicts Theorem~\ref{con-volume-nonexistence} (i), and hence for  $n= 3,4$, $p\in(-1,1)$ and $\lambda>1$, there exist $\varepsilon_0>0$ and $\Gamma>1$ depending on $n$, $p$ and $\lambda$ such that if, $n-\varepsilon_0<q< n+\varepsilon_0$ and a convex body $K\in \mathcal{K}^n_o$ satisfies that
\begin{equation}
\label{curvature-closeto-one-condagain}
\lambda^{-1}\mathcal{H}^{n-1}\leq \widetilde{C}_{p,q,K}\leq \lambda\mathcal{H}^{n-1},
\end{equation}
then ${\rm diam}\,K\leq \Gamma$. 

Finally, by considering the indirect assumption as in \eqref{Theorem1-volume} and the argument following \eqref{Theorem1-volume},  one can show the existence of an $\eta>0$ depending on $n$, $p$, $\Gamma$ and $\lambda$ such that if \eqref{curvature-closeto-one-condagain} holds for $K\in\mathcal{K}^n_o$, then $|K|\geq \eta$, completing the proof of Theorem~\ref{qcloseton-dualcurvcloseto1} when $n=3,4$ and $\lambda_m=\lambda>1$.
\proofbox

\section{The basic estimate for $o$-symmetric convex bodies when $-1<p<q<\min\{n,n+p\}$ and $q>0$}
\label{sec-basic-estimate-symmetric}

In this section, we prepare for the proof of Theorem~\ref{lambda-Cpq-symmetric} in Section~\ref{sec-lambda-Cpq-symmetric}. We start with two simple observations.

\begin{claim}
\label{hK-int-max}
For $\gamma>0$ and $n\geq 2$,
there exists $c>0$ depending  on $\gamma$ and $n$ such that if $K\in\mathcal{K}^n_o$, then
\begin{equation}
\label{hK-int-max-eq}
c\cdot\max_{v\in S^{n-1}}h_K(v)^\gamma\leq \int_{S^{n-1}}h_K^\gamma\,d\mathcal{H}^{n-1}\leq n\kappa_n\cdot\max_{v\in S^{n-1}}h_K(v)^\gamma.
\end{equation}
\end{claim}
\proof The upper bound in \eqref{hK-int-max-eq} readily holds. For the lower bound, let $x_0\in\partial K$ satisfy that $\|x_0\|=\max_{x\in K}\|x\|$. Then $K\subset \|w\|\,B^n$ for $w=x_0/\|x_0\|$, and hence $h_K(w)=\max_{v\in S^{n-1}}h_K(v)$. For $S_+^{n-1}=\{v\in S^{n-1}:\langle v,w\rangle\geq 0\}$, it follows that
$$
\int_{S^{n-1}}h_K^\gamma\,d\mathcal{H}^{n-1}\geq \int_{S_+^{n-1}}\langle v,x_0\rangle^\gamma\,d\mathcal{H}^{n-1}(v)=\int_{S_+^{n-1}}\langle v,w\rangle^\gamma\,d\mathcal{H}^{n-1}\cdot \|x_0\|^\gamma,
$$
proving the lower bound in \eqref{hK-int-max-eq}. 
\proofbox

For $q\in(0,1)$, Lemma~\ref{qth-intrvol-ellipsoid} presents some estimates that are dual to some
results by  Chen \cite{HCh18} about the so-called 
 $q$th intrinsic volume of an ellipsoid  (see also the recent preprint B\"or\"oczky, Kov\'acs, Mui, Zhang \cite{BKMZ} for  more exact estimates). 

\begin{lemma}
\label{qth-intrvol-ellipsoid} 
For $s\in(0,1)$ and $n\geq 2$, there exist $C_0,C_1,c_0>0$ depending on $s$ and $n$ with the following properties: 
 If $a_1\leq\ldots\leq a_n$ are the half axes of a centered ellipsoid $E\subset \R^n$, and $a_ne_n\in \partial E$ for $e_n\in S^{n-1}$, then
\begin{description}
\item{(i)} 
$\displaystyle  c_0 \cdot a_n^{-s}\leq \int_{S^{n-1}}h_E^{-s}\,d\mathcal{H}^{n-1}\leq C_0 \cdot a_n^{-s}$.

\item{(ii)} For $\xi\in(0,\frac12)$ and $\Gamma_\xi=\{x\in S^{n-1}:|\langle x,e_n\rangle|\leq \xi\}$, we have
\begin{equation}
\label{qth-intrvol-ellipsoid0}
 \int_{\Gamma_\xi}h_E^{-s}\,d\mathcal{H}^{n-1}\leq C_1 \cdot\xi^{\frac{1-s}{2-s}}\cdot a_n^{-s} .
\end{equation}
\end{description}
\end{lemma}
\proof The lower bound in (i) directly follows from fact that $h_E(v)\leq a_n$ for $v\in S^{n-1}$.

For the upper bound in (i) and for (ii), we observe that $(B^n+\R e_n)\cap \Sigma\subset \sqrt{2} B^n$, 
and
the use of polar coordinates in $\R^n$ implies that
$$
\int_{\sqrt{2}\,B^n}\|y\|^{s-n}\,dy=\frac{n\kappa_n}{s}\cdot 2^{\frac{s}2}=A.
$$
In particular, for any measurable $\omega\subset S^{n-1}$, we have
\begin{equation}
\label{Aomega}
\int_{\sqrt{2}\,B^n\cap {\rm pos}_+\omega}\|y\|^{s-n}\,dy=A\cdot \frac{\mathcal{H}^{n-1}(\omega)}{n\kappa_n}.
\end{equation}

Let $\Sigma=\{y\in \R^n:\,|\langle y,e_n\rangle|\leq 1\}$, and we observe that $h_E(x)\geq a_n|\langle x,e_n\rangle|$ for $x\in S^{n-1}$. Therefore,  for any measurable $\omega\subset S^{n-1}$, we deduce using polar coordinates that
\begin{align}
\nonumber
\int_{\omega}h_E^{-s}\,d\mathcal{H}^{n-1}\leq& \int_{\omega\backslash e_n^\bot}a_n^{-s}|\langle x,e_n\rangle|^{-s}\,d\mathcal{H}^{n-1}(x)\\
\nonumber
=& \frac{a_n^{-s}}{s}\int_{\omega\backslash e_n^\bot}\int_0^{|\langle x,e_n\rangle|^{-1}}r^{s-n}\cdot r^{n-1}\,drd\mathcal{H}^{n-1}(x)\\
\nonumber
=&\frac{a_n^{-s}}{s}\int_{\Sigma \cap {\rm pos}_+\omega}\|y\|^{s-n}\,dy\\
\label{inthK-s-omega}
\leq&
\frac{a_n^{-s}}{s}\left(\frac{A}{n\kappa_n}\cdot \mathcal{H}^{n-1}(\omega)+\right.\\
\nonumber
&\left.\int_{e_n^\bot\backslash B^n}\mathcal{H}^1\left(\{t:\,z+t e_n\in(\Sigma \cap {\rm pos}_+\omega) \}\right)\cdot \|z\|^{s-n}\,d\mathcal{H}^{n-1}(z)\right).
\end{align}
We deduce from \eqref{inthK-s-omega} with $\omega=S^{n-1}$ and using polar coordinates in $e_n^\bot$ that
\begin{align*}
\int_{S^{n-1}}h_E^{-s}\,d\mathcal{H}^{n-1}\leq &\frac{a_n^{-s}}{s}\left(A+\int_{e_n^\bot\backslash B^n}2\cdot \|z\|^{s-n}\,d\mathcal{H}^{n-1}(z)\right)\\
=&
\frac{a_n^{-s}}{s}\left(A+\frac{2(n-1)\kappa_{n-1}}{1-s}\right)
\end{align*}
verifying the upper bound in (i).

For (ii), we observe that for $\omega=\Gamma_\xi$ in \eqref{inthK-s-omega}, $\xi\leq \frac12$ implies that
\begin{equation}
\label{hK-s-narrow1}
\mathcal{H}^{n-1}(\Gamma_\xi)=(n-1)\kappa_{n-1}\int_{-\xi}^\xi (1-t^2)^{\frac{n-3}2}\,dt\leq
2(n-1)\kappa_{n-1} \left(\frac34\right)^{\frac{n-3}2}\cdot\xi.
\end{equation}
Next,  we set $\tau=\frac1{2-s}$, observe that if $z\in e_n^\bot$ with $1\leq \|z\|\leq \xi^{-\tau}$ and $z+te_n\in{\rm pos}_+\Gamma_\xi$, then
$$
|t|\leq \frac{\xi}{\sqrt{1-\xi^2}}\cdot \xi^{-\tau}\leq 2\xi^{1-\tau}=2\xi^{\frac{1-s}{2-s}}.
$$ 
We deduce that
\begin{align}
\nonumber
&\int_{(e_n^\bot\backslash B^n)\cap \xi^{-\tau}B^n}\mathcal{H}^1\left(\{t:\,z+t e_n\in(\Sigma \cap {\rm pos}_+\Gamma_\xi) \}\right)\cdot \|z\|^{s-n}\,d\mathcal{H}^{n-1}(z)\\
\label{hK-s-narrow2}
\leq &2\xi^{\frac{1-s}{2-s}}\cdot \int_{e_n^\bot\backslash B^n} \|z\|^{s-n}\,d\mathcal{H}^{n-1}(z)=\frac{2(n-1)\kappa_{n-1}}{1-s}\cdot 2\xi^{\frac{1-s}{2-s}}.
\end{align}
Finally, the last estimate we need for \eqref{inthK-s-omega} with $\omega=\Gamma_\xi$ is that if $\|z\|\geq \xi^{-\tau}$, the simply use that $\mathcal{H}^1\left(\{t:\,z+t e_n\in(\Sigma \cap {\rm pos}_+\Gamma_\xi) \}\right)\leq 2$, thus
\begin{align}
\nonumber
&\int_{e_n^\bot\backslash \xi^{-\tau}B^n}\mathcal{H}^1\left(\{t:\,z+t e_n\in(\Sigma \cap {\rm pos}_+\Gamma_\xi) \}\right)\cdot \|z\|^{s-n}\,d\mathcal{H}^{n-1}(z)\\
\nonumber
\leq &2 \int_{e_n^\bot\backslash \xi^{-\tau}B^n} \|z\|^{s-n}\,d\mathcal{H}^{n-1}(z)=\frac{2(n-1)\kappa_{n-1}}{1-s}\cdot \xi^{-\tau(s-1)}\\
\label{hK-s-narrow3}
=&\frac{2(n-1)\kappa_{n-1}}{1-s}\cdot \xi^{\frac{1-s}{2-s}}.
\end{align}
Combining  \eqref{inthK-s-omega}, \eqref{hK-s-narrow1}, \eqref{hK-s-narrow2} and \eqref{hK-s-narrow3} yields \eqref{qth-intrvol-ellipsoid0}, and in turn Lemma~\ref{qth-intrvol-ellipsoid}.
\proofbox

In the rest of section, we fix $\lambda>1$ and $p,q\in\R$ with $-1<p<q<\min\{n,n+p\}$, and consider only $o$-symmetric convex bodies $K\subset\R^n$ such that
\begin{equation}
\label{dual-curvature-bounded-cond}
\frac1{\lambda}\cdot\mathcal{H}^{n-1}\leq \widetilde{C}_{p,q,K}\leq \lambda\mathcal{H}^{n-1}
\end{equation}
on $S^{n-1}$. For two positive functions or quantities $a$ and $b$, we write $a\preceq b$ or $b\succeq a$ if there exists a constant $C\geq 1$ depending on $\lambda$, $n$, $p$, $q$ such that $a\leq Cb$, and $a\approx b$ if 
	$a\preceq b$ and $b\preceq a$. 
	
	For an $o$-symmetric convex body $K\subset\R^n$ satisfying \eqref{dual-curvature-bounded-cond},  let $E$ be the ellipsoid centered at the origin such that 
\begin{equation}
\label{KLS-ellipsoid-symmetric-containments}
E\subset K\subset nE,
\end{equation}
one can choose for example the John ellipsoid (cf. Schneider \cite{Sch14}), or the KLS ellipsoid of $K$ (cf. Kannan, Lovász, Simonovits \cite{KLS95}).
	Let $e_1,...,e_n$ form an orthonormal basis of $\R^n$ that form the principal directions for $E$ and let $r_1\leq\ldots \leq r_n$ hold for the half axes of $E$ where $r_ie_i\in \partial E$, $i=1,\ldots,n$. While $r_1,\ldots,r_n$ obviously depend on $K$, we do not signal this fact in the notation because $K$ will be clear from context.  In particular,
\begin{equation}
\label{diam-approx-rn00}
{\rm diam}\,K\approx r_n \mbox{ \ and \ }
|K|\approx r_1\cdot\ldots\cdot r_n.
\end{equation}

\begin{prop}
\label{rpqnvol}
For  $-1<p<q<\min\{n,n+p\}$, $q>0$ and $\lambda>1$, if an $o$-symmetric convex body $K\subset\R^n$ satisfies \eqref{dual-curvature-bounded-cond}, and $r_1,\ldots,r_n$  are the half axes of the KLS ellipsoid, then
\begin{equation}
\label{rpqnvol-eq}
r_n^{q-n-p}\cdot r_1\cdot\ldots\cdot r_n\approx 1.
\end{equation}
\end{prop}
\proof On the one hand, we deduce from \eqref{DhK-rhoK} and $q<n$ that $h_K(v)^{n-q}\leq \|Dh_K(v)\|^{n-q}$ for each $v\in S^{n-1}$ where $h_K$ is differentiable, thus applying first \eqref{diam-approx-rn00}, after that \eqref{cone-volume-from-CqK} in Lemma~\ref{Gauss-map-bijective} and \eqref{dual-curvature-bounded-cond}, and finally Claim~\ref{hK-int-max} and  \eqref{diam-approx-rn00} yield that
\begin{align*}
r_1\cdot\ldots\cdot r_n\approx& V_K(S^{n-1})\approx \int_{S^{n-1}} h_K^{p}\cdot \|Dh_K(v)\|^{n-q}\,d\mathcal{H}^{n-1}\\
\geq  &
\int_{S^{n-1}} h_K^{n+p-q}\,d\mathcal{H}^{n-1}\approx r_n^{n+p-q}.
\end{align*}
Therefore, to prove Proposition~\ref{rpqnvol}, all we are left to verify is that
\begin{equation}
\label{rpqnvol-upper}
r_1\cdot\ldots\cdot r_n\preceq  r_n^{n+p-q}.
\end{equation}
As $Dh_K(v)\in \partial K$ for $v\in S^{n-1}$, we have $ \|Dh_K(v)\|^{n-q}\leq (nr_n)^{n-q}$ (cf. \eqref{KLS-ellipsoid-symmetric-containments}), thus again  applying first \eqref{diam-approx-rn00},
after that \eqref{cone-volume-from-CqK} in Lemma~\ref{Gauss-map-bijective}, and finally \eqref{dual-curvature-bounded-cond} imply that
$$
r_1\cdot\ldots\cdot r_n\approx V_K(S^{n-1})\preceq \int_{S^{n-1}} h_K^{p}\cdot r_n^{n-q}\,d\widetilde{C}_{q,K}\preceq
r_n^{n-q}\cdot \int_{S^{n-1}} h_K^{p}\,d\mathcal{H}^{n-1},
$$
and hence \eqref{rpqnvol-upper} follows, if
\begin{equation}
\label{rpqnvol-hKp-upper}
\int_{S^{n-1}} h_K^{p}\,d\mathcal{H}^{n-1}\preceq  r_n^{p}.
\end{equation}
We observe that \eqref{rpqnvol-hKp-upper} readily holds by $h_K\leq nr_n$ (cf. \eqref{KLS-ellipsoid-symmetric-containments}) if $p\geq 0$, so we may assume that $-1<p<0$. In this case, $E\subset K$ and Lemma~\ref{qth-intrvol-ellipsoid} (i) imply that
$$
\int_{S^{n-1}} h_K^{p}\,d\mathcal{H}^{n-1}\preceq \int_{S^{n-1}} h_E^{p}\,d\mathcal{H}^{n-1}\preceq r_n^{p},
$$
 proving \eqref{rpqnvol-upper}, and in turn Proposition~\ref{rpqnvol}.
\proofbox

\section{Proof of Theorem~\ref{lambda-Cpq-symmetric}}
\label{sec-lambda-Cpq-symmetric}

For the convenience of the reader, we recall Theorem~\ref{lambda-Cpq-symmetric} as Proposition~\ref{lambda-Cpq-symmetric-prop}.

\begin{prop}
\label{lambda-Cpq-symmetric-prop}
For  $\lambda>1$, $-1<p<q<\min\{n,n+p\}$ with $q>0$, 
there exists $\Gamma>1$ depending on $n$, $\lambda$, $p$ and $q$, such that if an $o$-symmetric convex body $K\subset\R^n$ satisfies that  $\lambda^{-1}\mathcal{H}^{n-1}\leq \widetilde{C}_{p,q,K}\leq\lambda\mathcal{H}^{n-1}$,
then 
$$
{\rm diam}\,K\leq \Gamma\mbox{ \ and \ }|K|\geq 1/\Gamma.
$$
\end{prop}

\noindent{\it Proof of Proposition~\ref{lambda-Cpq-symmetric-prop}.}
	We suppose that the statement about the boundedness of the diameter in Proposition~\ref{lambda-Cpq-symmetric-prop} does not hold for some $-1<p<q<\min\{n,n+p\}$, $q>0$ and $\lambda>1$,
 and seek a contradiction. 
It follows that  for some sequence $K_m\subset\R^n$ of $o$-symmetric convex bodies such that $\lim_{m\to\infty}{\rm diam}\,K_m=\infty$, we have
\begin{equation}
\label{control-shape-lambda-theoproof00}
\frac1{\lambda}\,\mathcal{H}^{n-1}\leq\widetilde{C}_{p,q,K_m}\leq \lambda \mathcal{H}^{n-1}.
\end{equation}

Let $E_{K_m}$ be the origin centered KLS ellipsoid of John ellipsoid of $K_m$, and hence
\begin{equation}
\label{KLS-ellipsoid-symmetric}
\E_{K_m}\subset K_m\subset nE_{K_m},
\end{equation} 
and let $r_{1,m}\leq\ldots\leq r_{n,m}$ be the half axes of $E_{K_m}$.
As the $L_p$  dual curvature measure is invariant under orthogonal transformations, we may also assume that
 there exists a fixed orthonormal basis $e_1,\ldots,e_n$ of $\R^n$ such that $r_{i,m}e_i\in\partial K_m$ for each $m$ and $i=1,\ldots,n$.
It follows from ${\rm diam}\,K_m\leq 2nr_{n,m}$ (cf. \eqref{KLS-ellipsoid-symmetric})  and $\lim_{m\to\infty}{\rm diam}\,K_m= \infty$ that
\begin{equation}
\label{rnm-to-infty-symmetric}
\lim_{m\to\infty}r_{n,m}= \infty.
\end{equation}
According to \eqref{rnm-to-infty-symmetric}, we may assume that for some $k=\{1,\ldots,n\}$ and $\theta\in(0,1)$, we have
\begin{equation}
\label{control-shape-rkn-nuKo00}
\begin{array}{rcl}
r_{n-k+1,m}&\geq& \theta\cdot r_{n,m},\\[1ex]
{\displaystyle\lim_{m\to\infty}}\frac{r_{n-k,m}}{r_{n,m}}&=&0
\end{array}
\end{equation} 
where the second condition in \eqref{control-shape-rkn-nuKo00} is void if $k=n$.

For two positive functions or quantities $a$ and $b$, we write $a\preceq b$ or $b\succeq a$ if there exists a constant $C\geq 1$ depending on $\lambda$, $n$, $p$, $q$ such that $a\leq Cb$, and $a\approx b$ if 
	$a\preceq b$ and $b\preceq a$. 
We deduce  from the basic estimate \eqref{rpqnvol-eq} that
\begin{equation}
\label{rpqnmvol-eq}
r_{n,m}^{q-n-p}\cdot r_{1,m}\cdot\ldots\cdot r_{n,m}\approx 1.
\end{equation}

We claim that actually
\begin{equation}
\label{k-atmost-n-100}
k\leq n-1.
\end{equation}
We suppose that \eqref{k-atmost-n-100} does not hold - that is, $r_{i,m}\geq \theta\cdot r_{n,m}$ for $i=1,\ldots,n-1$, - and seek a contradiction. In this case, \eqref{rnm-to-infty-symmetric} yields that
$$
\lim_{m\to\infty}r_{n,m}^{q-n-p}\cdot r_{1,m}\cdot\ldots\cdot r_{n,m}\geq
\theta^{n-1}\lim_{m\to\infty}r_{n,m}^{q-p}=\infty,
$$
contradicting \eqref{rpqnmvol-eq}, and verifying \eqref{k-atmost-n-100}.

For any $m$, let
\begin{align*}
X_m=&\left(L^\bot+\frac{1}{2n}(E_m\cap L)\right)\cap \partial E_m,\\
Z_m=&\partial K_m\cap {\rm pos}_+X_m,
\end{align*}
We observe that if $z\in Z_m$, then $z=tx$ for an $x\in X_m$ where $ t\leq n$ (cf. \eqref{KLS-ellipsoid-symmetric}), thus
$z|L\in \frac{1}{2}(E_m\cap L)$, and turn \eqref{control-shape-rkn-nuKo00} yield that
\begin{equation}
\label{Zm-proj-symmetric}
(z|L)+ \frac{\theta}{2}\cdot  r_{n,m}(B^n\cap L)\subset (z|L)+ \frac{1}{2}(E_m\cap L)\subset E_m\cap L\subset K_m.
\end{equation}
In addition, we deduce that from $q<n$ and \eqref{KLS-ellipsoid-symmetric} that if $z\in Z_m$, then
\begin{equation}
\label{zin-widetildeZm-size}
\|z\|^{q-n}\geq  (nr_{n,m})^{q-n}.
\end{equation}

Let $\omega_m=\nu_{K_m}(Z_m)$, and hence $V_{K_m}(\omega_m)=\widetilde{V}_{K_m}(Z_m)$ by Lemma~\ref{Gauss-map-bijective}. We also observe that $\widetilde{V}_{K_m}(Z_m)\geq \widetilde{V}_{E_m}(X_m)$ by $E_m\subset K_m$ and \eqref{tildeVKZ-def2}. Using the positive definite diagonal transform $\Phi_m$ such that $\Phi_mE_m=B^n$, we deduce from  the equivariance   \eqref{tildeVK-linear-inv} of the cone volume measure under linear transforms
that
$$
\frac{V_{K_m}(\omega_m)}{|E_m|}=\frac{\widetilde{V}_{K_m}(Z_m)}{|E_m|}\geq 
\frac{\widetilde{V}_{E_m}(X_m)}{|E_m|}=\frac{\widetilde{V}_{B^n}(\{v\in S^{n-1}:\,\|v|L\|\leq \frac1{2n}\})}{|B^n|}\succeq 1,
$$
and hence \eqref{KLS-ellipsoid-symmetric} implies that
\begin{equation}
\label{VKmomegam-size}
V_{K_m}(\omega_m)\succeq r_{1,m}\cdot\ldots\cdot r_{n,m}.
\end{equation}

For $v\in \omega_m$, if $v$ is the exterior normal at a $z\in Z_m$, then
 $z=(z|L)+\tau a$ and $v=w\cos\alpha+b\sin\alpha$ for some $a,b\in S^{n-1}\cap L^\bot$, $w\in S^{n-1}\cap L$ and $\alpha\in[0,\frac{\pi}2]$ where $\tau\leq nr_{n-k,m}$ by \eqref{KLS-ellipsoid-symmetric}. Since $y=(z|L)+\frac{\theta}{2}\cdot  r_{n,m}\cdot w\in K_m$ by \eqref{Zm-proj-symmetric}, it follows that
\begin{align*}
0\leq&\langle v,z-y\rangle=\sin \alpha\cdot \tau \langle b,a\rangle-
\left\langle\cos\alpha\cdot w,\frac{\theta}{2}\cdot  r_{n,m}\cdot w\right\rangle\\
\leq  &nr_{n-k,m}-\cos\alpha\cdot \frac{\theta}{2}\cdot  r_{n,m};
\end{align*}
therefore, if $v\in \omega_m$, then
$$
\|v|L\|=\cos \alpha\leq \frac{2nr_{n-k,m}}{\theta  r_{n,m}}=\xi_m.
$$
It follows from  \eqref{control-shape-rkn-nuKo00} that
\begin{equation}
\label{nuKmz-proj-symmetric}
\omega_m\subset \Gamma_{\xi_m}=\left\{v\in S^{n-1}:\,\langle v,e_n\rangle\leq \xi_m\right\}
\mbox{ \ where }\lim_{m\to\infty}\xi_m=0.
\end{equation}

Next, we deduce from  Lemma~\ref{Gauss-map-bijective} that
$$
\int_{\omega_m}h_{K_m}^{p}\,d\widetilde{C}_{p,q,K_m}=n\int_{\omega_m}(\varrho_K\circ\alpha^*_K)^{q-n}\,dV_{K_m},
$$
and hence the condition \eqref{control-shape-lambda-theoproof00} and the basic estimate \eqref{rpqnmvol-eq} imply that
\begin{equation}
\label{pq-core-estimate-symmetric}
\int_{\omega_m}\left(\frac{h_{K_m}}{nr_{n,m}}\right)^{p}\,d\mathcal{H}^{n-1}\approx
\frac1{r_{1,m}\cdot\ldots\cdot r_{n,m}}
\int_{\omega_m}\left(\frac{\varrho_K\circ\alpha^*_K}{nr_{n,m}}\right)^{q-n}\,dV_{K_m}.
\end{equation}

On the right hand side of \eqref{pq-core-estimate-symmetric}, we apply first \eqref{zin-widetildeZm-size}, and after that \eqref{VKmomegam-size} to obtain that
\begin{equation}
\label{pq-core-estimate-symmetric-RHS}
\frac1{r_{1,m}\cdot\ldots\cdot r_{n,m}}
\int_{\omega_m}\left(\frac{\varrho_K\circ\alpha^*_K}{nr_{n,m}}\right)^{q-n}\,dV_{K_m}\geq
\frac{V_{K_m}(\omega_m)}{r_{1,m}\cdot\ldots\cdot r_{n,m}}\succeq 1.
\end{equation}
On the other hand, let $s=\max\{\frac12,-p\}\in(0,1)$. Since $h_{K_m}(v)\leq nr_{n,m}$ (cf. \eqref{KLS-ellipsoid-symmetric}) for $v\in S^{n-1}$, we deduce from Lemma~\ref{qth-intrvol-ellipsoid} (ii) and \eqref{nuKmz-proj-symmetric} that on the left hand side of \eqref{pq-core-estimate-symmetric}, we have
\begin{align}
\nonumber
\limsup_{m\to\infty}\int_{\omega_m}\left(\frac{h_{K_m}}{nr_{n,m}}\right)^{p}\,d\mathcal{H}^{n-1}\leq &
\limsup_{m\to\infty}\int_{\omega_m}\left(\frac{h_{K_m}}{nr_{n,m}}\right)^{-s}\,d\mathcal{H}^{n-1}\\
\nonumber
\leq &\limsup_{m\to\infty}(nr_{n,m})^{s}\int_{ \Gamma_{\xi_m}}h_{E_m}^{-s}\,d\mathcal{H}^{n-1}\\
\label{pq-core-estimate-symmetric-LHS}
\preceq&\limsup_{m\to\infty}r_{n,m}^{s}\cdot \xi_m^{\frac{1-s}{2-s}}\cdot r_{n,m}^{-s}=0.
\end{align}
Therefore, combining \eqref{pq-core-estimate-symmetric-RHS} and \eqref{pq-core-estimate-symmetric-LHS} contradicts \eqref{pq-core-estimate-symmetric}. We conclude
the existence of $\Gamma>1$ depending on $n$, $\lambda$, $p$ and $q$, such that if an $o$-symmetric convex body $K\subset\R^n$ satisfies that  $\lambda^{-1}\mathcal{H}^{n-1}\leq \widetilde{C}_{p,q,K}\leq\lambda\mathcal{H}^{n-1}$,
then ${\rm diam}\,K\leq \Gamma$.

Finally, by considering the indirect assumption as in \eqref{Theorem1-volume} and the argument following \eqref{Theorem1-volume},  one can show the existence of an $\eta>0$ depending on $n$, $p$, $\Gamma$ and $\lambda>1$ such that if $\lambda^{-1}\mathcal{H}^{n-1}\leq \widetilde{C}_{p,q,K}\leq\lambda\mathcal{H}^{n-1}$ holds for an $o$-symmetric convex body $K\subset\R^n$, then $|K|\geq \eta$, completing the proof of Proposition~\ref{lambda-Cpq-symmetric-prop}.
\proofbox

\section{Proof of Theorems~\ref{local-uniqueness-Cpq} and \ref{local-uniqueness-even-Cpq}}
\label{secUniqueness}

In this section, we complete the proofs of Theorem~\ref{local-uniqueness-Cpq} and Theorem~\ref{local-uniqueness-even-Cpq}. The argument is similar to that of \cite[Proof of Theorem 1.1]{CFL22} and \cite[Proof of Theorem 1.2]{BoS24}; for the reader’s convenience, we sketch the proof here.

We need the weak convergence of the $L_p$ dual Minkowski measure in our setting.

\begin{lemma}
\label{Cpq-weak-convergence-qton}
Let $\lambda>1$ and $p\in(-1,1)$, and let $q_m>1$ tend to $n$, and $K_m\in \mathcal{K}^n_0$ tend to $K_\infty\in \mathcal{K}^n_0$ as $m$ tends to infinity where $\lambda^{-1}\mathcal{H}^{n-1}\leq \widetilde{C}_{p,q_m,K_m}\leq \lambda\mathcal{H}^{n-1}$.
Then $\widetilde{C}_{p,q_m,K_m}$ tends weakly to $S_{p,K_\infty}$.
\end{lemma}
\begin{proof}
The main idea is similar to the argument leading to \eqref{phiVKinfty-limitK-VPsiKm}. 
To simplify the formulas, let $d_m=q_m-n$. As $\widetilde{C}_{p,n,K_m}=S_{p,K_m}$ tends weakly to $S_{p,K_\infty}$ according to
\eqref{weak-convergence-SpK}, we may assume that $d_m\neq 0$ for any $m$.

Now Lemma~\ref{Gauss-map-bijective} (i) says that  $h_{K_m}$ is $C^{1}$  on the set $\{h_{K_m}>0\}$. 
For a continuous function $\varphi:S^{n-1}\to S^{n-1}$, we claim that
\begin{align}
\label{phiVKinfty-limitK-VPsiKm0}
\int_{S^{n-1}}\varphi\,dS_{p,K_\infty}=&\lim_{m\to \infty} \int_{S^{n-1}}\left(d_m^2h_{K_m}^2+\|Dh_{K_m}\|^2\right)^{\frac{d_m}2}
\varphi h_{K_m}^{-p}\,dV_{K_m}\\
\label{phiVKinfty-limitK-SPsiKm0}
=&\lim_{m\to \infty} \frac1n\int_{S^{n-1}}\left(d_m^2h_{K_m}^2+\|Dh_{K_m}\|^2\right)^{\frac{d_m}2}
\varphi h_{K_m}^{1-p}\,dS_{K_m}.
\end{align}
We observe that \eqref{phiVKinfty-limitK-VPsiKm0} and \eqref{phiVKinfty-limitK-SPsiKm0} are equivalent by definition. 
Choosing some $R>0$ such that $K_\infty\subset {\rm int}\,RB^n$, we deduce from \eqref{DhK-length-rhoK}, \eqref{DhK-rhoK} and $d_m\to 0$  that if $m$ is large, then 
$$
d_m^2h_{K_m}^2\leq d_m^2h_{K_m}^2+\|Dh_{K_m}\|^2\leq R^2,
$$
 and hence
$$
Q_m=h_{K_m}^{1-p}\left(d_m^2h_{\Psi_mK_m}^2+\|Dh_{K_m}\|^2\right)^{\frac{d_m}2}
$$
satisfies that  
$$
\begin{array}{rclcll}
d_m^{d_m}h_{K_m}^{1+d_m-p}&\leq &Q_m &\leq & R^{d_m}h_{K_m}^{1-p}&\mbox{ if }d_m>0\\[1ex]
R^{d_m}h_{K_m}^{1-p}&\leq &Q_m &\leq & |d_m|^{d_m}h_{K_m}^{1+d_m-p}&\mbox{ if }d_m<0.
\end{array}
$$
As $h_{K_m}$ tends uniformly to $h_{K_\infty}$, Corollary~\ref{htodelta-bound} implies that $Q_m$ tends uniformly to $h_{K_\infty}^{1-p}$. Now $S_{K_m}$ tends weakly to $S_{K_\infty}$ (cf. \eqref{weak-convergence-SK}), proving \eqref{phiVKinfty-limitK-SPsiKm0} by \eqref{weak-convergence-SK}. In turn, we conclude the claim \eqref{phiVKinfty-limitK-VPsiKm0}.

Since $\|Dh_{K_m}(x)\|\geq h_{K_m}(x)>0$ for  $x\in S^{n-1}\backslash\nu_{K_m}(o)$ according to \eqref{DhK-rhoK} and 
$V_{K_m}\left(\nu_{K_m}(o)\right)=0$, we deduce from \eqref{phiVKinfty-limitK-VPsiKm0} and \eqref{CpqK-from-cone-volume} in Lemma~\ref{Gauss-map-bijective} that
\begin{align}
\label{phiSpKinfty-limCpqKm}
\int_{S^{n-1}}\varphi\,dS_{p,K_\infty}=&\lim_{m\to \infty} \int_{S^{n-1}}\theta_m\cdot\varphi\,d\widetilde{C}_{p,q_m,K_m};&&\\
\nonumber
\theta_m(x)=&\left(\frac{d_m^2h_{K_m}(x)^2+\|Dh_{K_m}(x)\|^2}{\|Dh_{K_m}(x)\|^2}\right)^{\frac{d_m}2}&&\mbox{if }x\in S^{n-1}\backslash\nu_{K_m}(o);\\
\nonumber
\theta_m(x)=&1&&\mbox{if }x\in \nu_{K_m}(o)
\end{align}
where we mean $\nu_{K_m}(o)=\emptyset$ if $o\in{\rm int}\,K_m$.
Since $2^{\frac{-|d_m|}2}\leq \theta_m(x)\leq 2^{\frac{|d_m|}2}$ hold for any $x\in S^{n-1}$, it follows that $\theta_m$ tends uniformly to $1$ on $S^{n-1}$. Therefore, \eqref{phiSpKinfty-limCpqKm} yields that
$$
\int_{S^{n-1}}\varphi\,dS_{p,K_\infty}=\lim_{m\to \infty} \int_{S^{n-1}}\varphi\,d\widetilde{C}_{p,q_m,K_m},
$$
proving that $\widetilde{C}_{p,q_m,K_m}$ tends weakly to $S_{p,K_\infty}$.
\end{proof}

Now to prepare the proof of Theorem~\ref{local-uniqueness-Cpq}, let
\[
C^{2,\alpha}_{>0}(S^{n-1}) :=  \{u \in C^{2,\alpha}(S^{n-1}) : u > 0 \text{ on } S^{n-1}\}
\]
and
\[
C^{2,\alpha}_{>0, e} :=  \{u \in C^{2,\alpha}_{>0}(S^{n-1}) : u \text{ is even}\}.
\]

Consider the map
\[
F: \mathbb{R} \times C^{2,\alpha}_{>0}(S^{n-1}) \longrightarrow \mathbb{R} \times C^{\alpha}(S^{n-1}),
\]
defined by
\[
F(q, u) = \left(q, \left(\|\nabla u\|^2 + u^2\right)^{\frac{q-n}{2}} u^{1-p} \det\left(\nabla^2 u + u\,I\right)\right).
\]
Let \(q_0 = n\) and \(u_0 = 1\). A straightforward computation shows that the Fréchet derivative of \(F\) at \((q_0, u_0)\) is given by
\[
T: \mathbb{R} \times C^{2,\alpha}(S^{n-1}) \longrightarrow \mathbb{R} \times C^{\alpha}(S^{n-1})
\]
with
\[
T(t, \phi) = (t,\, \mathcal{L}_{u_0} \phi),
\]
where
\[
\mathcal{L}_{u_0}: C^{2,\alpha}(S^{n-1}) \longrightarrow C^{\alpha}(S^{n-1})
\]
is the linear mapping defined by
\[
\mathcal{L}_{u_0}(\phi) = \Delta_{S^{n-1}} \phi + (n-p)\phi.
\]

\bigskip

\begin{proof}[Proof of Theorem~\ref{local-uniqueness-Cpq}] We  recall that the eigenvalues of \(-\Delta_{S^{n-1}}\) are given by
\[
\lambda_k = k^2 + (n-2)k,\quad k = 0, 1, 2, \dots,
\]
so that $\mathcal{L}_{u_0}$ is invertible when $-1 < p < 1$. This implies that \(T\) is invertible. By the inverse function theorem, there exists a small neighborhood \(\mathcal{N}\) of \((q_0, u_0)\) such that if \((q_1,u_1)\) and \((q_2,u_2)\) are in \(\mathcal{N}\) and
\[
F(q_1, u_1) = F(q_2, u_2),
\]
then \(u_1 = u_2\). Moreover, since \(\mathcal{N}\) is a neighborhood of \((q_0, u_0)\), there exists a constant \(\delta_0 > 0\) such that if \((q, u) \notin \mathcal{N}\) then either \(|q - q_0| > \delta_0\) or \(\|u - u_0\|_{C^{2,\alpha}(S^{n-1})} > \delta_0\).

It now suffices to verify the claim that there exists a constant \(\eta \in (0,1)\), depending only on \(n\), \(\alpha\), and \(p\), such that if \(|q - n| < \eta\) and \(f \in C^{\alpha}(S^{n-1})\) satisfies \(\|f - 1\|_{C^{\alpha}(S^{n-1})} < \eta\), then for any convex body $K\in\mathcal{K}^n_0$ satisfying
\[
d\widetilde{C}_{p,q,K} = f\,dx,
\]
the support function $h_K$ satisfies 
\begin{equation}
\label{qhK-in-N}
(q, h_K) \in \mathcal{N}.
\end{equation}

Suppose, to the contrary, that there exist a sequence \(q_m \to n\), functions \(f_m \to 1\) in \(C^{\alpha}(S^{n-1})\), and convex bodies $K_m\in\mathcal{K}^n_0$ such that
\begin{equation}\label{limitcontra1}
d\widetilde{C}_{p,q_m,K_m} = f_m\,dx,
\end{equation}
and
\begin{equation}\label{limitcontra2}
\|h_{K_m} - 1\|_{C^{2,\alpha}(S^{n-1})} > \delta_0.
\end{equation}

By Theorem \ref{qcloseton-dualcurvcloseto1}, we have \(\|h_{K_m}\|_{L^\infty(S^{n-1})} \leq C\) and \(|K_m| > C^{-1}\) for some constant \(C\) independent of \(m\) (provided \(m\) is sufficiently large). By the Blaschke selection theorem, up to a subsequence, \(K_m\) converges in the Hausdorff distance to a convex body \(K_\infty\). Since \(\widetilde{C}_{p,q_m,K_m}\) tends weakly to $S_{p,K_\infty}$  by Lemma~\ref{Cpq-weak-convergence-qton}, and $f_m$ tends uniformly to the constant $1$ function, we have
\[
S_{p,K_\infty} = \mathcal{H}^{n-1}.
\]
According to  Brendle, Choi, Daskalopoulos \cite{BCD17}, it follows that \(K_\infty\) is the unit ball; hence, \(h_{K_\infty} \equiv 1\).

Thus, we deduce that
\[
\|h_{K_m} - 1\|_{L^\infty(S^{n-1})} \to 0 \quad \text{as } m \to \infty.
\]
Since \(f_m \to 1\) in \(C^{\alpha}(S^{n-1})\) and by applying Caffarelli's regularity theory and Schauder estimates (similar to the arguments in \cite[Proof of Lemma 4.1]{CFL22}), we obtain that
\[
\|h_{K_m} - 1\|_{C^{2,\alpha}(S^{n-1})} \to 0 \quad \text{as } m \to \infty,
\]
which contradicts \eqref{limitcontra2}. 

Therefore, we conclude the claim \eqref{qhK-in-N}, and in turn  Theorem \ref{local-uniqueness-Cpq}. 
\end{proof}

\begin{proof}[Proof of Theorem~\ref{local-uniqueness-even-Cpq}]
The proof is similar to that of Theorem~\ref{local-uniqueness-Cpq}. We only need to restrict \(F\) to \(\mathbb{R} \times C^{2,\alpha}_{>0, e}(S^{n-1})\). Note that for the extraction of a convergent subsequence, instead of using Theorem \ref{qcloseton-dualcurvcloseto1}, we use Theorem \ref{lambda-Cpq-symmetric}. Note also that the limiting equation here is 
\[
\widetilde{C}_{p,q,K} = \mathcal{H}^{n-1},
\]
and the uniqueness of even solutions has been proved by Chen, Huang, Zhao \cite{CHZ19}, provided \(p \geq -n\) and \(q \leq \min\{n, n+p\}\) with \(q \neq p\). 
\end{proof}

\section*{Acknowledgment}
\addcontentsline{toc}{section}{Acknowledgment}
Research of B\"or\"oczky was supported by Hungarian National Research and Innovation Office grant ADVANCED\_24 150613, and research of Chen was supported by National Key $R\&D$ program of China 2022YFA1005400, National Science Fund for Distinguished Young Scholars (No. 12225111).


\begin{thebibliography}{99}

\bibitem{Ale38a}
A.D. Aleksandrov: 
On the theory of mixed volumes. III. Extension of two theorems of
Minkowski on convex polyhedra to arbitrary convex bodies.
(Russian; German summary) Mat. Sbornik N.S., 3 (1938), 27-46. 

\bibitem{Ale96}
A.D. Aleksandrov:  
Selected works. Part I. Gordon and Breach Publishers, Amsterdam, 1996.

\bibitem{And99}
 B. Andrews:
Gauss curvature flow: the fate of rolling stone.
 Invent. Math., 138 (1999), 151-161.

\bibitem{AGN16}
B. Andrews, Pengfei Guan, Lei Ni:
Flow by the power of the Gauss curvature.
Adv. Math., 299 (2016), 174-201.


\bibitem{BG} 
F. Barthe, O. Gu\'{e}don, S. Mendelson, A. Naor:
A probabilistic approach to the geometry of the $l_{p}^{n}$-ball.
Ann. of Probability, 33 (2005), 480-513.

\bibitem{BBCY19}
G. Bianchi, K.J. B\"or\"oczky, A. Colesanti, Deane Yang:
The $L_p$-Minkowski problem for $-n< p<1$ according to Chou-Wang.
Adv. Math., 341 (2019), 493-535.

\bibitem{BBC20}
G. Bianchi, K.J. B\"or\"oczky, A. Colesanti: 
Smoothness in the $L_p$ Minkowski problem for $p<1$. 
J. Geom. Anal., 30 (2020), 680-705.

\bibitem{Bor23}
K.J. B\"or\"oczky:
The Logarithmic Minkowski conjecture and the $L_p$-Minkowski Problem.
In: Harmonic analysis and convexity,  De Gruyter, Berlin, (2023), 83-118.
arxiv:2210.00194
		
\bibitem{BoF19}
K.J. B\"or\"oczky, F. Fodor: 
The $L_p$ dual Minkowski problem for $p>1$ and $q>0$. 
J. Differential Equations, 266 (2019), 7980-8033.

\bibitem{BoK22}
K.J. B\"or\"oczky, P. Kalantzopoulos:
Log-Brunn-Minkowski inequality under symmetry.
Trans. AMS,  375 (2022), 5987-6013.

\bibitem{BKMZ}
 K.J. B\"or\"oczky, \'A. Kov\'acs, Stephanie Mui, Gaoyong Zhang:
Dual Curvature Density Equation with Group Symmetry. submitted.
arxiv:2503.10044

\bibitem{BLYZ12}
K.J. B\"or\"oczky, E. Lutwak, Deane Yang, Gaoyong Zhang:
The log-Brunn-Minkowski-inequality.
Adv. Math., 231 (2012), 1974-1997.

\bibitem{BLYZ13}
K.J. B\"or\"oczky, E. Lutwak, Deane Yang, Gaoyong Zhang:
The Logarithmic Minkowski Problem.
Journal of the American Mathematical Society, 26 (2013), 831-852.

		
\bibitem{BoS24} 
K.J. B\"or\"oczky, C. Saroglou:
Uniqueness when the $L_p$ curvature is close to be a constant for $p\in[0,1)$.
Calc. Var. Partial Differential Equations 63 (2024), no. 6, Paper No. 154, 26 pp.

\bibitem{BCD17} 
S. Brendle, Kyeongsu Choi, P. Daskalopoulos: 
Asymptotic behavior of flows by powers of the Gaussian curvature. 
Acta Math., 219 (2017), 1-16.

\bibitem{CMH}
C. Cabezas-Moreno, Jinrong Hu:
The $L_p$ dual Christoffel-Minkowski problem for $1<p<q\leq k+1$ with $1\leq k\leq n$.
arXiv:2504.04931  

		
\bibitem{Caf90a}
L.A. Caffarelli:
A localization property of viscosity solutions to the Monge-Amp\`ere equation and their strict convexity.
Ann. of Math. (2) 131, (1990), 129-134.
		
\bibitem{Caf90b}
L.A. Caffarelli:
Interior $W^{2,p}$ estimates for solutions of the Monge-Amp\`ere equation.
Ann. of Math. (2), 131 (1990), 135-150.

\bibitem{Caf93}
L.A. Caffarelli:
A note on the degeneracy of convex solutions to Monge-Amp\`ere equation.
Comm. Partial Differential Equations, 18 (1993), 1213-1217.

\bibitem{Cal58}
E. Calabi:
Improper affine hyperspheres of convex type and a generalization of a theorem by K. J\"orgens.
Mich. Math. J., 5 (1958), 105-126.

\bibitem{CHZ19}
Chuanqiang Chen, Yong Huang, Yiming Zhao: 
Smooth solutions to the $L_p$ dual Minkowski problem. 
Math. Ann., 373 (2019), 953-976.
		
\bibitem{HCh18}
Haodi Chen: 
On a generalised Blaschke-Santal\'o inequality. 
arXiv:1808.02218

\bibitem{CFL22}
Shibing Chen, Yibin Feng, Weiru Liu: Uniqueness of solutions to the logarithmic Minkowski problem in $\mathbb{R}^3$. Adv. Math., 411 (A) (2022), 108782.

\bibitem{CHLZ}
Shibing Chen, Shengnan Hu, Weiru Liu, Yiming Zhao:
On the planar Gaussian-Minkowski problem.
arXiv:2303.17389

\bibitem{CHLL20}
Shibing Chen, Yong Huang, Qi-Rui Li, Jiakun Liu:
The $L_p$-Brunn-Minkowski inequality for $p<1$.
Adv. Math., 368 (2020), 107166.	






\bibitem{ChL21}
Haodi Chen,  Qi-Rui Li: 
The $L_p$ dual Minkowski problem and related parabolic flows. 
J. Funct. Anal. 281 (2021), Paper No. 109139, 65 pp.  

\bibitem{CLZ17}
Shibing Chen, Qi-Rui Li, Guangxian Zhu:
On the $L_p$ Monge-Amp\`ere equation.
Journal of Differential Equations, 263 (2017), 4997-5011.

	
\bibitem{CLZ19}
Shibing Chen, Qi-Rui Li, Guangxian Zhu:
The Logarithmic Minkowski Problem for non-symmetric measures.
Trans. Amer. Math. Soc., 371 (2019), 2623-2641.

\bibitem{ChY76} 
Shiu-Yuen Cheng, Shing-Tung Yau:
On the regularity of the solution of the $n$-dimensional Minkowski problem.
Comm. Pure 	Appl. Math. 29 (1976), 495-561.

\bibitem{ChY86} 
Shiu-Yuen Cheng, Shing-Tung Yau:
Complete affine hypersurfaces. Part I. The completeness of affine metrics.
Commun. Pure Appl. Math., 39 (1986), 839-866.


\bibitem{ChW06} 
Kai-Shen  Chou, Xu-Jia Wang:
The $L_{p}$-Minkowski
problem and the Minkowski problem in centroaffine geometry.
Adv. Math., 205 (2006), 33-83.

\bibitem{CLM17}
A. Colesanti, G. Livshyts, A. Marsiglietti: On the stability of Brunn-Minkowski type inequalities. 
J. Funct. Anal. 273 (2017), 1120-1139.

\bibitem{CoL20}
A. Colesanti, G. Livshyts: 
A note on the quantitative local version of the log-Brunn-Minkowski inequality. 
The mathematical legacy of Victor Lomonosov-operator theory, 85-98, Adv. Anal. Geom., 2, De Gruyter, Berlin, 2020.

\bibitem{CrF23}
G. Crasta, I. Fragal\`a:
On a geometric combination of functions related to Pr\'ekopa-Leindler inequality.
Mathematika, 69 (2023), 482-507.

\bibitem{DPF14}
G. De Philippis, Alessio Figalli:
The Monge-Amp\`ere equation and its link to optimal transportation.
Bull. Amer. Math. Soc. 51 (2014), 527--580.

\bibitem{Du21}
Shi-Zhong  Du: 
On the planar $L_p$-Minkowski problem. 
Jour. Diff. Equations, 287 (2021), 37-77.

\bibitem{FHX23}
Yibin Feng,  Shengnan Hu, Lei Xu: 
On the $L_p$ Gaussian Minkowski problem.
J. Differential Equations, 363 (2023), 350-390. 

\bibitem{FLX23}
Yibin Feng, Weiru Liu, Lei Xu: 
Existence of non-symmetric solutions to the Gaussian Minkowski problem. 
 J. Geom. Anal., 33 (2023), no. 3, Paper No. 89, 39 pp. 

		
\bibitem{Fig17}
A. Figalli:
The Monge-Amp\`ere equation and its applications. 
Z\"urich Lectures in Advanced Mathematics. EMS,  Z\"urich, 2017.

\bibitem{Fir74}
W.J. Firey: 
Shapes of worn stones. 
Mathematika, 21 (1974), 1-11.


\bibitem{GHWXY19}
R. Gardner, D. Hug, W. Weil, Sudan Xing, Deping Ye: 
General volumes in the Orlicz-Brunn-Minkowski theory and a related Minkowski problem I. 
Calc. Var. Partial Differential Equations 58 (2019), Paper No. 12, 35 pp.

\bibitem{GHXY20}
R. Gardner, D. Hug, Sudan Xing, Deping Ye: 
General volumes in the Orlicz-Brunn-Minkowski theory and a related Minkowski problem II. 
Calc. Var. Partial Differential Equations 59 (2020), Paper No. 15, 33 pp.

\bibitem{GromovMilman}
M. Gromov, V.D. Milman:
Generalization of the spherical isoperimetric inequality for uniformly convex Banach Spaces. 
Composito Math., 62 (1987), 263-282.

\bibitem{LGWa}
Qiang Guang, Qi-Rui Li, Xu-Jia Wang:
The $L_p$-Minkowski problem with super-critical exponents.
arXiv:2203.05099

\bibitem{LGW23}
Qiang Guang, Qi-Rui Li, Xu-Jia Wang:
Q. Guang, Q-R. Li, X.-J. Wang:
Flow by Gauss curvature to the $L_p$ dual Minkowski problem. 
Math. Eng. 5 (2023), no. 3, Paper No. 049, 19 pp. 

\bibitem{JinrongHu}
Jinrong Hu:
Uniqueness of solutions to the dual Minkowski problem. arxiv:2502.18032

\bibitem{HuI25}
Yingxiang Hu, M.N. Ivaki: 
Stability of the Cone-volume measure with near constant density. 
 Int. Math. Res. Not. IMRN, 6 (2025), rnaf062. 

\bibitem{HuL13}
Yong Huang, QiuPing Lu:  
On the regularity of the $L_p$ Minkowski problem. 
Adv. in Appl. Math., 50 (2013), 268-280.

\bibitem{HLYZ16}
Yong Huang, E. Lutwak, Deane Yang, Gaoyong Zhang: 
Geometric measures in the dual Brunn-Minkowski theory and their associated Minkowski problems. 
Acta Math. 216 (2016),  325-388.

\bibitem{HXZ21} 
Yong Huang, Dongmeng Xi, Yiming Zhao: 
The Minkowski problem in Gaussian probability space.
Adv. Math., 385 (2021), 107769.

\bibitem{HuZ18} 
Yong Huang, Yiming Zhao: 
On the Lp dual Minkowski problem. 
Adv. Math. 332 (2018), 57-84. 
	
\bibitem{Hug96}
D. Hug: 
Contributions to affine surface area. 
Manuscripta Math., 91 (1996), 283-301.

\bibitem{HLYZ05} 
D. Hug, E. Lutwak, Deane Yang, Gaoyong Zhang:
On the $L_{p}$ Minkowski problem for polytopes.
Discrete Comput. Geom., 33 (2005), 699-715.

\bibitem{Iva22}
M.N. Ivaki:
On the stability of the $L_p$-curvature.
JFA, 283 (2022), 109684.

\bibitem{IvM23}
M.N. Ivaki, E. Milman:
Uniqueness of solutions to a class of isotropic curvature problems.
Adv. Math., 435 (2023), part A, Paper No. 109350, 11 pp. 

\bibitem{IvM23b}
M.N. Ivaki, E. Milman:
$L^p$-Minkowski Problem under Curvature Pinching.
Int. Math. Res. Not., IMRN, accepted.
arXiv:2307.16484

\bibitem{JLW15}
Huaiyu Jian, Jian Lu, Xu-Jia Wang:  
Nonuniqueness of solutions to the $L_p$-Minkowski problem. 
Adv. Math. 281 (2015), 845-856.

\bibitem{KLS95}
R. Kannan, L. Lovász, M. Simonovits: 
Isoperimetric problems for convex bodies and a localization lemma. 
Discrete Comput. Geom., 13 (1995),  541-559.

\bibitem{Kol20}
A.V. Kolesnikov:
Mass transportation functionals on the sphere with applications to the logarithmic Minkowski problem.
Mosc. Math. J., 20 (2020), 67-91.

\bibitem{KoL}
A.V. Kolesnikov, G. V. Livshyts:
On the Local version of the Log-Brunn-Minkowski conjecture and some new related geometric inequalities.
 Int. Math. Res. Not. IMRN, (2022), no. 18, 14427-14453.

\bibitem{KoM22}
A.V. Kolesnikov, E. Milman:
Local $L_p$-Brunn-Minkowski inequalities for $p<1$.
 Memoirs of the American Mathematical Society, 277 (2022), no. 1360.

\bibitem{LiW}
Haizhong Li, Yao Wan:
Classification of solutions for the planar isotropic $L_p$ dual Minkowski problem.
arXiv:2209.14630


\bibitem{LLL22}
Qi-Rui Li, Jiakun Liu, Jian Lu:
Non-uniqueness of solutions to the dual $L_p$-Minkowski problem.
IMRN, (2022), 9114-9150.

\bibitem{LSW20}
Qi-Rui Li, Weimin Sheng, Xu-Jia Wang: Flow by Gauss curvature to the Aleksandrov and dual Minkowski problems. J. Eur. Math. Soc. (JEMS), 22 (2020), 893-923. 


\bibitem{Liu22}
JiaQian Liu:
The $L_p$-Gaussian Minkowski problem. 
Calc. Var. Partial Differential Equations 61 (2022), Paper No. 28, 23 pp. 

\bibitem{LMNZ20}
G. Livshyts, A. Marsiglietti, P. Nayar, A. Zvavitch: 
On the Brunn-Minkowski inequality for general measures with applications to new isoperimetric-type inequalities. Trans. Amer. Math. Soc. 369 (2017), no. 12, 8725-8742.

\bibitem{LuP21}
Fangxia Lu, Zhaonian Pu:  
The $L_p$ dual Minkowski problem about $0<p<1$ and $q>0$. 
Open Math., 19 (2021), 1648-1663.

\bibitem{Lud10}
M. Ludwig:
General affine surface areas.
Adv. Math., 224 (2010), 2346-2360.

\bibitem{Lut90}
E. Lutwak: Centroid bodies and dual mixed volumes. In: Proc. London Math. Soc. 60 (1990),
365-391.

\bibitem{Lut93}
E. Lutwak:
Selected affine isoperimetric inequalities.
In: Handbook of convex geometry,
North-Holland, Amsterdam, 1993, 151-176.

\bibitem{Lut93a}
E. Lutwak:
The Brunn-Minkowski-Firey theory. I. Mixed volumes and the Minkowski problem.
J. Differential Geom. 38 (1993), 131-150.

\bibitem{LYZ18}
E. Lutwak, Deane  Yang, Gaoyong Zhang:
$L_p$ dual curvature measures.
Adv. Math., 329 (2018), 85-132.

\bibitem{Mil24}
E. Milman:
A sharp centro-affine isospectral inequality of
Szeg\H{o}-Weinberger type and the $L_p$-Minkowski problem. J. Differential Geom. 127 (2024), 373-408.

\bibitem{Mil25}
E. Milman:
Centro-Affine Differential Geometry and the Log-Minkowski Problem.
 J. Eur. Math. Soc. (JEMS), 27 (2025), 709-772.

\bibitem{Min03}
H. Minkowski:
Volumen und Oberf\"ache.
Math. Ann., 57 (1903), 447-495.

\bibitem{Min11}
H. Minkowski:  
Theorie der konvexen K\"orper, insbesondere Begr\"undung ihres
Oberfl\"achenbegriffs, manuscript, 
printed in: Gesammelte Abhandlungen 2,  ed.: D. Hilbert, Teubner, Leipzig 1911, 131-229.


\bibitem{Nar07}
A. Naor:
The surface measure and cone measure on the sphere of $l^n_p$.
	Trans. Amer. Math. Soc., 	359 (2007), 1045-1079.

\bibitem{Nir53} 
L. Nirenberg:
The Weyl and Minkowski problems in differential geometry in the large.
 Comm. Pure and Appl. Math., 6 (1953), 337-394.

\bibitem{PaW12} 
G. Paouris, E. Werner:
Relative entropy of cone measures and $L_{p}$ centroid bodies.
Proc. London Math. Soc., 104 (2012), 253-286.

\bibitem{Pog72}
A.V. Pogorelov:
On the improper convex affine hyperspheres.
Geom. Dedic., 1 (1972), 33-46.

\bibitem{Pog78} 
A.V. Pogorelov:
The Minkowski multidimensional problem.
V.H. Winston \& Sons, Washington, D.C, 1978.

\bibitem{Sar15}
C. Saroglou:
Remarks on the conjectured log-Brunn-Minkowski inequality.
Geom. Dedicata 177 (2015), 353-365.

\bibitem{Sar16}
C. Saroglou:
More on logarithmic sums of convex bodies. 
Mathematika, 62 (2016), 818-841.

\bibitem{Sar22}
C. Saroglou:
On a non-homogeneous version of a problem of Firey. 
Math. Ann., 382 (2022), 1059-1090.

\bibitem{Sar}
C. Saroglou:
A non-existence result for the $L_p$-Minkowski problem.	Proc. AMS, accepted.
arXiv:2109.06545

\bibitem{Sch14}
R. Schneider:
Convex Bodies: The Brunn-Minkowski Theory.
Cambridge University Press, 2014.

\bibitem{Stancu}
A. Stancu:
The discrete planar $L_0$-Minkowski problem.
Adv. Math., 167 (2002), 160-174.
	
\bibitem{Stancu1}
A. Stancu:
On the number of solutions to the discrete two-dimensional $L_0$-Minkowski problem.
Adv. Math., 180 (2003), 290-323.

\bibitem{Sta22}
A. Stancu: Prescribing centro-affine curvature from one convex body to another.
 Int. Math. Res. Not. IMRN, (2022), 1016-1044.

\bibitem{TrW08}
N.S. Trudinger, Xu-Jia Wang: The Monge-Amp\`ere equation and its geometric applications. 
Handbook of geometric analysis. No 1, Adv. Lect. Math. (ALM), 7, Int. Press,  (2008), 467-524.

\bibitem{vHa}
R. van Handel:
The local logarithmic Brunn-Minkowski inequality for zonoids.
In: Geom. Aspects of Funct. Anal., Lecture Notes in Mathematics 2327, Springer,  (2023), 355-379. 

		
\end{thebibliography}
\end{document}